\PassOptionsToPackage{lowtilde}{url}
\documentclass[12pt]{amsart}
\usepackage{amsthm,amssymb}
\usepackage{tikz-cd}
\usepackage{todonotes}
\usepackage{accents}
\usepackage{chngcntr, csquotes}
\usepackage[margin=1.1in]{geometry}
\usepackage[colorlinks]{hyperref}
\usepackage{cleveref}
\usepackage{enumitem}
\usepackage[mode=build]{standalone}
\usetikzlibrary{decorations.pathreplacing}

\newcommand\blfootnote[1]{%
  \begingroup
  \renewcommand\thefootnote{}\footnote{#1}%
  \addtocounter{footnote}{-1}%
  \endgroup
}

\newcounter{proofcount}
\AtBeginEnvironment{proof}{\stepcounter{proofcount}}
\newtheorem{claim}{Claim}
\makeatletter
\@addtoreset{claim}{proofcount}
\makeatother

\theoremstyle{remark}
\makeatletter
\newtheorem*{cproof/}{Proof of Claim \rev@cproofmark} 
\newenvironment{cproof}[1][\@nil]
  {\def\@tmp{#1}%
   \ifx\@tmp\@nnil
       \def\rev@cproofmark{\theclaim}%
    \else
       \let\rev@cproofmark\@tmp%
    \fi
   \pushQED{\qed}\begin{cproof/}}
  {\popQED\end{cproof/}}
\makeatother

\newcommand{\rr}{\mathbb{R}}
\newcommand{\qq}{\mathbb{Q}}

\newcommand{\VL}{V \mathord{=} L}
\newcommand{\CH}{\textsf{CH}}
\newcommand{\hh}{\mathcal{H}}

\newcommand{\fcant}{2^{<\omega}}
\newcommand{\cant}{2^{\omega}}

\newcommand{\len}{\ell}

\newcommand{\str}[1]{\overline{#1}}
\newcommand{\real}[1]{\tilde{#1}}
\newcommand{\realo}[1]{\tilde{#1}_{+}}

\newcommand{\set}[2]{\{ #1 \,\vert\, #2 \}}
\newcommand{\Set}[2]{\left\{ #1 \,\middle\vert\, #2 \right\}}

\theoremstyle{plain}
\newtheorem{thm}{Theorem}[section] 
\crefname{thm}{theorem}{theorems} 
\newtheorem{prp}[thm]{Proposition}
\newtheorem{lem}[thm]{Lemma}
\newtheorem{cor}[thm]{Corollary}

\theoremstyle{definition}
\newtheorem{dfn}[thm]{Definition}

\theoremstyle{remark}
\newtheorem{rem}[thm]{Remark}

\numberwithin{equation}{section}

\DeclareMathOperator{\Dim}{Dim}
\DeclareMathOperator{\ran}{ran}

\DeclareMathOperator{\proj}{proj}
\DeclareMathOperator{\diam}{diam}

\begin{document}

\title[Co-analytic Counterexamples to Marstrand's Theorem]{Co-analytic Counterexamples to Marstrand's Projection Theorem}

\author{Linus Richter}
\address{School of Mathematics and Statistics, Victoria University of Wellington, Wellington, New Zealand}
\email{linus.richter@sms.vuw.ac.nz}
\urladdr{\url{https://homepages.ecs.vuw.ac.nz/~richtelinu/}}

\blfootnote{

2012 \textit{ACM classification}: Theory of Computation $\rightarrow$ Complexity theory and logic

}


\begin{abstract}
Assuming $\VL$, we construct a plane set $E$ of Hausdorff dimension $1$ whose every orthogonal projection onto straight lines through the origin has Hausdorff dimension $0$. This is a counterexample to J.\ M.\ Marstrand's seminal projection theorem \cite{marstrand}. While counterexamples had already been constructed decades ago, initially by R.\ O.\ Davies \cite{daviesCounterexample}, the novelty of our result lies in the fact that $E$ is co-analytic. Following Marstrand's original proof \cite{marstrand} (and R.\ Kaufman's newer, and now standard, approach \cite{kaufman} based on capacities), a counterexample to the projection theorem cannot be analytic, hence our counterexample is optimal. We then extend the result in a strong way: we show that for each $\epsilon \in (0,1)$ there exists a co-analytic set $E_{\epsilon}$ of dimension $1+ \epsilon$, each of whose orthogonal projections onto straight lines through the origin has Hausdorff dimension $\epsilon$. The constructions of $E$ and $E_{\epsilon}$ are by induction on the countable ordinals, applying a theorem by Z.\ Vidny\'{a}nszky \cite{zoltan}.
\end{abstract}

\maketitle


\section{Introduction}

Hausdorff measure is one of the fundamental notions of classical fractal geometry. While Lebesgue measure limits itself to well-behaved structures such as planes or lines, or remains uninformative otherwise, Hausdorff measure is much finer (in particular, it extends Lebesgue measure). Hausdorff measure is defined similarly to Lebesgue measure: one considers covers of a set. However, for Hausdorff measure, there are two important differences: (1) the covers need not be open; and (2), the diameter of the cover sets is raised to a chosen power $\alpha > 0$. It is easily seen that for every set $X$ there exists $s \geq 0$ such that the $\alpha$-dimensional Hausdorff measure of $X$ is 0 if $\alpha > s$; and that it is infinite if $\alpha < s$. What value the $\alpha$-dimensional Hausdorff measure of $X$ takes with $\alpha = s$ is anyone's guess. It could be 0, infinite, or take on some other finite value. For this reason, $s$ is called the critical value---it is also called the \emph{Hausdorff dimension} of $X$.

An illuminating quote explaining the intuition behind the critical value is the following \cite[p.\ 168]{edgar}:

\begin{quote}
This idea of dimension is an abstraction of what we already know from elementary geometry. If $A$ is a [...] curve, then its length is a useful way to measure its size; but its ``area'' and ``volume'' are $0$. The dimensions $2$ and $3$ are too large to help in measuring the size of $A$. If $B$ is the surface of a sphere, then its area is positive and finite. We can say its ``length'' is infinite (for example, since it contains curves that are as long as we like which spiral around); its ``volume'' is 0, since it is contained in a solid spherical shell whose thickness is as small as we like. So for the set $B$, the dimension 1 is too small, the dimension 3 is too large, and the dimension 2 is just right. The $\alpha$-dimensional Hausdorff measure give us a way of measuring the size of a set for dimensions $\alpha$ other than the integers 1, 2, 3, \ldots
\end{quote}

Since the early development of fractal geometry, a few theorems have turned out to be foundational. Marstrand's projection theorem \cite{marstrand} dating back to 1954 is one of them: it forms a cornerstone of classical fractal geometry today. While ignored for decades (with the term ``fractal geometry'' only arriving in the 1970s), fractal geometry, and projection theorems like Marstrand's, are researched intensively nowadays \cite{falconerFraserXiong}.

In essence, the theorem states that orthogonal projections of analytic sets cannot drop too far in dimension---but there are exceptions. The first to notice this was R.\ O.\ Davies \cite{daviesCounterexample}, who constructed a counterexample non-constructively, assuming the Continuum Hypothesis $\CH$. Simplifications \cite{kaufman} and generalisations \cite{mattila} of Marstrand's theorem followed over the subsequent decades. Nowadays, the standard argument to prove Marstrand's projection theorem is Kaufman's proof \cite{kaufman} which is based on \emph{energy potential characterisations} of Hausdorff dimension. Refinements of the theorem are sought after today \cite{dimensionZero,strongMarstrand}.\\

While fractal geometry traditionally depended on measure arguments, other approaches have advanced over the last decades. With the first such being R.\ Kaufman's aforementioned proof \cite{kaufman} of Marstrand's Theorem, more recently, advances in bridging fractal geometry to information theoretical tools have succeeded \cite{lutzGales,mayordomo,hitchcockPhD,lutz2003}: standing out is the recent \emph{point-to-set principle} of Lutz and Lutz \cite{kakeya}, which relates the Hausdorff dimension of a \emph{set} to the dimensions of its constituent \emph{points}. The dimension of points, in this context, is given using tools of algorithmic information theory: one considers computability theory, and in particular, Kolmogorov randomness. We present the required background in subsequent sections.

The point-to-set principle has already proven to be a very useful tool in the theory of orthogonal projections. This applies to both Hausdorff dimension $\dim_H$ as well as packing dimension $\dim_P$: Lutz and Stull \cite{effectiveDimension} have shown using algorithmic arguments that if $X \subset \rr^2$ satisfies $\dim_H(X) = \dim_P(X)$ then Marstrand's theorem applies, and so the requirement of being analytic can be dropped. They also give a new bound on the packing dimension of orthogonal projections under packing dimension, and provide a new proof of Marstrand's theorem in the same paper. It is noted that it was already known that packing dimension does not admit a Marstrand-like result \cite{packingNoMars}.\\

In the present paper we show that being analytic is in fact sharp for Marstrand's theorem, a fact previously unknown: we construct an optimal counterexample to Marstrand's projection theorem using non-classical tools. In particular, we use Lutz' and Lutz' point-to-set principle \cite{kakeya} to construct a set of Hausdorff dimension 1, all of whose projections vanish. Using a result by Vidny\'{a}nszky \cite{zoltan} we show that the constructed set is co-analytic. We then extend our result to produce a co-analytic Marstrand-failing set of dimension $1 + \epsilon$ for each $\epsilon \in (0,1)$ in a strong way: for each $\epsilon$ we produce a co-analytic set $X$ such that $\dim_H(X) = 1 + \epsilon$ while its projection onto every line through the origin has dimension $\epsilon$, the minimal allowable value.\\

Throughout our work we make the set-theoretic assumption that $\VL$. This is for combinatorial reasons, as it allows us to argue by transfinite induction in our construction, following Vidny\'{a}nszky's theorem.\\

We would like to note that the results in this paper have independently been obtained by T.\ Slaman and D.\ Stull.

\subsection{Marstrand's Projection Theorem}
The original statement of John Marstrand's Projection Theorem is as follows (we have flipped the order from how it appears in the source).

\begin{thm}[MPT, {\cite[Thm I \& II]{marstrand}}]\label{thm:mpt}
	Let $E \subset \rr^2$ be analytic. Let $\proj_{\theta}(E)$ denote the projection of $E$ onto the unique line passing through the origin at angle $\theta$ with the first axis. Let $\mu$ denote the one-dimensional Lebesgue measure.
	\begin{enumerate}
		\item If $\dim_H(E) \leq 1$ then for almost all $\theta \in [0,\pi)$ we have $\dim_H(\proj_{\theta}(E)) = \dim_H(E)$. \label{eq:mpt1}
		\item If $\dim_H(E) > 1$ then for almost all $\theta \in [0,\pi)$ we have $\mu(\proj_{\theta}(E)) > 0$. \label{eq:mpt2}
	\end{enumerate}
\end{thm}

It is noted that \cref{eq:mpt2} is strictly stronger than an assertion about Hausdorff dimension \cite{falconerEtAl}. This follows from the standard fact that, in $\rr^m$, Hausdorff measure generalises $m$-dimensional Lebesgue measure (the proof is a straightforward measure argument \cite[6.1]{edgar}). In particular, if $A \subset \rr$ has positive Lebesgue measure, then $\dim_H(A) = 1$. An easy Corollary of \cref{eq:mpt2} is hence:

\begin{cor}\label{cor:mpt2}
	If $E \subset \rr^2$ is analytic and $\dim_H(E) > 1$ then for almost all $\theta \in [0,\pi)$ we have $\dim_H(\proj_{\theta}(E)) = 1$.
\end{cor}

\subsection{Our Theorems}
We provide co-analytic counterexamples to both \cref{eq:mpt1,eq:mpt2} of \cref{thm:mpt}. Since analytic sets satisfy \cref{thm:mpt}, our results are sharp. The first theorem is proven in \cref{sec:firstThm}, the second in \cref{sec:secondThm}. Recall that we assume $\VL$.

\theoremstyle{plain}
\newtheorem*{thm:firstThm}{\Cref{thm:firstThm}}
\begin{thm:firstThm}
There exists a co-analytic set $E \subset \rr^2$ such that $\dim_H(E) = 1$ while, for every $\theta \in [0,2\pi)$ we have $\dim_H(\proj_{\theta}(E)) = 0$.
\end{thm:firstThm}

The construction is carried out via a theorem due to Z.\ Vidny\'{a}nszky \cite{zoltan}, which allows us to build co-analytic sets using co-analytic conditions by recursion. Let $\mathbb{D}$ denote the first quadrant of the unit disc, which we shall consider in more detail in \cref{sec:polar}.

Given $F \subset \mathbb{D}^{\leq \omega} \times [0,\pi/2] \times \mathbb{D}$, a set $X = \set{x_{\alpha}}{\alpha < \omega_1}$ is \emph{compatible with $F$} if the following exist:
\begin{itemize}
	\item an enumeration $\set{p_{\alpha}}{\alpha < \omega_1}$ of $B$; and
	\item an enumeration $\set{A_{\alpha}}{\alpha < \omega_1} \subset \mathbb{D}^{\leq\omega}$ such that if $\alpha < \omega_1$ then $A_{\alpha} = X \upharpoonright \alpha$
\end{itemize}
such that for each $\alpha < \omega_1$ we have $(A_{\alpha},p_{\alpha},x_{\alpha}) \in F$.

\theoremstyle{plain}
\newtheorem*{thm:zoltansThm}{\Cref{thm:zoltansThm}}
\begin{thm:zoltansThm}[{\cite{zoltan}}]
	Let $F \subset \mathbb{D}^{\leq \omega} \times [0,\pi/2] \times \mathbb{D}$. If $F$ is co-analytic and if for all $(A,p) \in \mathbb{D}^{\leq \omega} \times [0,\pi/2]$ the section
\begin{align*}
     F(A,p) = \{ x \in \mathbb{D} \mid F(A,p,x) \}                                                                                                                                                                                                   \end{align*}
is cofinal in the Turing degrees, then there exists a co-analytic set $X \subset \mathbb{D}$ that is compatible with $F$.
\end{thm:zoltansThm}

We revisit Vidny\'{a}nszky's theorem in more detail in \cref{sec:zoltan}.

Our second contribution is the following stronger result, which we prove in \cref{sec:secondThm}.

\theoremstyle{plain}
\newtheorem*{thm:secondThm}{\Cref{thm:secondThm}}
\begin{thm:secondThm}
For every $\epsilon \in (0,1)$, there exists a co-analytic set $E \subset \rr^2$ such that $\dim_H(E) = 1 + \epsilon$ while, for every $\theta \in [0,2\pi)$ we have $\dim_H(\proj_{\theta}(E)) = \epsilon$.
\end{thm:secondThm}

The arguments in the proof of this more general theorem are similar to those prior; however, the coding procedure is more involved.


\subsection*{Acknowledgments}

I would like to thank my doctoral advisor Daniel Turetsky for bringing the question around counterexamples to Marstrand's theorem to my attention, as well as for his continuous advice provided throughout the development of this paper, and for his very informative feedback on earlier versions of this manuscript.

\section{Preliminaries}

In this section, we give some necessary preliminaries on Hausdorff measure and dimension, as well as computability theory and the theory of complexity. These are required in order to use the point-to-set principle \cref{thm:pts} in our construction.

\subsection{A review of Hausdorff measure and dimension}
While classically geometric measure theory formed the backbone of fractal geometry, this is not necessarily so today. When working with the point-to-set principle, measure theory is mostly unnecessary. However, we occasionally switch between one and the other; some proofs are expressed and proven more simply in the language of measures than in that of computability theory. Hence we provide a very brief introduction to Hausdorff measure, dimension, and in particular Lipschitz maps below.

Consider a subset $E \subset \rr^2$. In order to define Hausdorff measure, we first consider
\begin{align*}
 \hh^s_{\delta}(E) = \inf\left\{ \sum_{i < \omega} |U_i|^s \ \middle\vert \ E \subset \bigcup_{i < \omega} U_i \land (\forall i < \omega)(|U_i| < \delta) \right\}
\end{align*}
where $d$ is the usual Euclidean distance and $|U| = \sup \{ d(x,y) \mid x,y \in U \}$, the \emph{diameter of $U$}.

\begin{rem}\label{HausdorffMeasureIncreases}
Observe that as $\delta$ increases, we get to include more covers in our infimum, and hence
\begin{center}
 if $0 < \delta < \delta'$ then $\hh^s_{\delta'} < \hh^s_{\delta}$.
\end{center}
So, as $\delta$ decreases, the term $\hh^s_{\delta}(E)$ increases. In particular, the limit $\lim_{\delta \rightarrow 0^+} \hh^s_{\delta}(E)$ always exists (it might be infinite!).
\end{rem}

\begin{dfn}
Let $E \subset \rr^2$. We define the \emph{$s$-dimensional Hausdorff measure of $E$} as follows:
\begin{align*}
\hh^s(E) = \lim_{\delta \rightarrow 0^+} \hh^s_{\delta}(E).
\end{align*}
\end{dfn}

It is easily seen that there must exist a \emph{critical value} for $s$ at which the $s$-dimensional Hausdorff measure changes from $\infty$ to 0---this is the \emph{Hausdorff dimension}.

\begin{dfn}\label{dfn:hausdorffDim}
 Let $E \subset \rr^2$. Then the \emph{Hausdorff dimension $\dim_H(E)$} is defined as
 \begin{align*}
  \dim_H(E) = \sup\{ s \geq 0 \mid \mathcal{H}^s(E) = \infty \} = \inf\{ s \geq 0 \mid \mathcal{H}^s(E) = 0 \}.
 \end{align*}
\end{dfn}

Finally, and quite importantly, Hausdorff dimension is well-behaved under Lipschitz maps. Recall that a map $f \colon \rr^m \rightarrow \rr^m$ is \emph{Lipschitz with constant $M > 0$} if for all $x,y \in \rr^m$ we have $|f(x) - f(y)| \leq M|x-y|$. (Here, $| \cdot |$ denotes the Euclidean norm on $\rr^m$.) Fixing $m=2$ it is noted that Lipschitz maps cannot increase dimension:

\begin{lem}\label{lem:rotationPreservesDimH}
Let $E \subset \rr^2$. If $f \colon \rr^2 \rightarrow \rr^2$ satisfies a Lipschitz condition then $\dim_H(f(E)) \leq \dim_H(E)$.	
\end{lem}

We use the following simple lemma.

\begin{lem}\label{lem:boundOnHausdorffMeasure}
Suppose $f \colon \rr^2 \rightarrow \rr^2$ is Lipschitz with constant $M > 0$. Then
\[ 
	\mathcal{H}^s(f(E)) \leq M^s\hh^s(E)
\]	
for all $s \geq 0$.
\end{lem}

\begin{proof}
	Suppose $f$ has Lipschitz constant $M$, and that $(U_i)$ is a $\delta$-cover for $E$. Then $(f(U_i))$ is an $M\delta$-cover for $f(E)$. Thus
	\[
		\hh^s_{M\delta}(f(E)) \leq \sum_{i < \omega} |f(U_i)|^s \leq \sum_{i < \omega} M^s|U_i|^s
	\]
	since, for any $f(x),f(y) \in f(U_i)$ we have $|f(x) - f(y)| \leq M|x-y|$, whence $|f(U_i)| \leq M|U_i|$ follows after taking suprema. Finally, taking $\lim_{\delta \rightarrow 0^+}$ above yields the result.
\end{proof}

\begin{proof}[Proof of Lemma \ref{lem:rotationPreservesDimH}]
	This follows straight from the lemma above: suppose $\dim_H(E) = s$, and assume $f$ is Lipschitz with constant $M > 0$. Then
	\[ 
		\hh^s(f(E)) \leq M^s\hh^s(E) < \infty.
	\]
	Since $\dim_H(E) = \sup \{ s \mid \hh^s(E) = \infty \}$ we have $\hh^s(f(E)) < \infty$. Therefore $\dim_H(f(E)) \leq s$, as needed.
\end{proof}

As a special (but extremely useful) case we note that if $f$ is an isometry then Hausdorff measure is in fact fixed. This yields that:

\begin{cor}\label{cor:rotationPreservesDimH}
Hausdorff dimension is preserved under isometries. In particular, it is preserved under rotation and translation.
\end{cor}

\subsection{Kolmogorov complexity}
Our main arguments rely on the notion of information density, which we describe in terms of \emph{prefix-free Kolmogorov complexity}. Standard, and very comprehensive, references for Kolmogorov complexity are Downey and Hirschfeldt \cite{downey} as well as Li and Vit\'{a}ny \cite{livitanyi}. A very short but very readable introduction is Fortnow \cite{fortnow}. We give a brief account of the most relevant notions and some history below.\\

We choose to express the basic notions of algorithmic information theory using Turing machines. In particular, we pick \emph{universal prefix-free} machine as our \emph{reference machine} $U$. Such a machine exists by construction: the meat of the argument lies in the standard result that every prefix-free p.c.\ (\emph{partial computable}) function has a prefix-free machine that computes it (e.g.\ \cite[3.5]{downey}).
Our universal machine $U$ satisfies the following: every prefix-free p.c.\ function $f \colon \fcant \rightarrow \fcant$ has a program code $p_f$ for which $U(p_f,x) = f(x)$. The machine $U$ has two tapes: an \emph{input tape} and a \emph{condition tape} (it also has a work tape, but this will be of no consequence to our arguments, so we shall ignore this fact). Define a p.c.\ function $h$ such that whenever $p_f$ is a prefix-free program for a prefix-free p.c.\ function $f$ then
\[
	h\left( 0^{|p_f|}1p_fx \right) = U(p_f,x) = f(x).
\]
Observe that $h$ is prefix-free. If $A \in \cant$ is an oracle, let $U^A$ be the universal prefix-free machine that has access to the oracle $A$ (this means, the machine can perform a step of the type ``does $k$ belong to $A$?'' for any $k < \omega$, and branch accordingly), and define $h_A$ analogously. Further, if $\tau \in \fcant$ then let $U^\tau$ be the machine with $\tau$ written on the condition tape.\\

We introduce the following notation: if $\sigma \in \fcant$ let $\len(\sigma)$ denote the length of $\sigma$. To avoid confusion with the absolute value on $\rr$ which we use heavily in later section, we reserve the symbol $|\cdot|$ for the space $\rr$.

\begin{dfn}
Let $\sigma \in \fcant$. The \emph{Kolmogorov complexity of $\sigma$} is
\begin{align*}
	K(\sigma) &= \min\set{\len(\rho)}{h(\rho) = \sigma}.
\intertext{If $\tau \in \fcant$ then the \emph{conditional Kolmogorov complexity} of $\sigma$ given $\tau$ is}
	K(\sigma \mid \tau) &= \min\set{\len(\rho)}{h_\tau(\rho) = \sigma}.
\end{align*}
If $A$ is an oracle, $K^A(\sigma)$ is defined analogously, with $h_A$ in place of $h_{\tau}$. In that case, the condition tape is empty (accessing the oracle does not require any tapes), and hence we can define $K^A(\sigma \mid \tau)$ as above.
\end{dfn}

From now on, we need to take the prefix-freeness of all machines into account. Note that all $\log$ in this paper are $\log_2$. As a result of prefix-freeness we immediately obtain that $K$ is \emph{subadditive up to a constant}: for all strings $\sigma,\tau \in \fcant$ we have $K(\sigma\tau) \leq K(\sigma) + K(\tau) + c$. Similarly, the following fact is basic yet useful.

\begin{lem}
	Let $\sigma \in \fcant$. Then $K(\sigma) \leq \len(\sigma) + 2\log(\len(\sigma)) + c$ for some constant $c$.
\end{lem}

\subsection{Randomness}
We give a brief introduction to the basic notions of randomness. For in-depth reviews see Downey and Hirschfeldt \cite[6.2]{downey} and Li and Vit\'{a}nyi \cite[3.5]{livitanyi}. We give one important result here: fortunately, the two definitions given above are equivalent. (This is not the only equivalence: Martin-L\"of \cite{martinloef}, Schnorr \cite{schnorr}, Levin \cite{levin}, and Chaitin \cite{chaitin} all yield equivalent notions of randomness. This is strong evidence that the notion is, in fact, ``correct''.)\\

The central objects of discussion in Kolmogorov complexity on $\cant$ comprise \emph{random sequences}: elements of $\cant$ that defy short descriptions up to a constant for all of its initial segments. This definition goes back to Kolmogorov \cite{kolmogorov} and Solomonoff \cite{solomonoff} who simultaneously and independently developed the early theory of algorithmic information. Another definition of randomness for sequences uses statistical (measure-theoretic) tests, the Martin-L\"of-tests (or ML-tests), and goes back to Martin-L\"of \cite{martinloef}.

Recall that an open subset of $\cant$ is the countable union of basic clopen sets, and that every basic clopen set is determined by some $\sigma \in \fcant$. A set $U \subset \cant$ is \emph{(lightface) $\Sigma_1^0$}, or \emph{effectively open}, if the set of finite strings that determine $U$ is computably enumerable. A sequence $\mathcal{U} = \set{U_n}{n < \omega}$ of sets $U_n \subset \cant$ is \emph{uniformly $\Sigma^0_1$} if every $U_n$ is $\Sigma_1^0$ witnessed by some c.e.\ set $V_n \subset \fcant$, and the sequence $\mathcal{V} = \set{V_n}{n < \omega}$ is uniformly c.e.\ itself. Martin-L\"of randomness uses the natural Lebesgue measure generated by the clopen sets: if $ \sigma \in \fcant$ then $\lambda(\set{\sigma f}{f \in \cant}) = 2^{-\len(\sigma)}$. An ML-test is then a computably enumerable sequence of effectively open sets $U_n$ of decreasing diameter. A \emph{Martin-L\"of-random sequences} is not captured by any ML-test.

The following proposition connects the two notions.

\begin{prp}
	Let $f \in \cant$. The following notions are equivalent:
	\begin{enumerate}
	\item $f \in 2^{\omega}$ is Kolmogorov random: there exists a constant $c$ such that for all $n < \omega$ we have $K(f \upharpoonright n) \geq n - c$;
	\item $f$ is Martin-L\"of-random: for any uniformly $\Sigma_1^0$ sequence $\mathcal{U} = \set{U_n}{n < \omega}$ of sets $U_n \subset \cant$ for which $\lambda(U_n) \leq 2^n$ we have $f \not\in \bigcap \mathcal{U}$ (so $f$ \emph{passes the test}).
\end{enumerate}
\end{prp}

These results relativise (see \cite[6.4]{downey}).

\begin{dfn}
	Let $A \in \cant$ be an oracle. A string $f \in \cant$ is \emph{Kolmogorov random relative to $A$} if there exists a constant $c$ such that for all $n < \omega$ we have $K^A(f \upharpoonright n) \geq n-c$.
\end{dfn}

It is a standard result that a universal ML-test $\mathcal{V}$ exists, from which it follows immediately that the Kolmogorov random sequences have measure 1 in $\cant$: every ML-test is Lebesgue null, and hence $\cant \setminus \mathcal{V}$ has measure 1. If $f$ passes the universal test $\mathcal{V}$, it passes all ML-tests, and hence $\cant \setminus \mathcal{V}$ is the set of ML-randoms. Thus every $\sigma \in \fcant$ has a Kolmogorov random extension. Further, if $A \in \cant$ is an oracle, then we can construct a universal ML-test relative to $A$. Hence every $\sigma \in \fcant$ has a Kolmogorov random extension relative to $A$.

\subsection{Coding objects \texorpdfstring{in $\fcant$}{using finite strings}}\label{sec:codings}
While, formally, our arguments take place in $\fcant$, we naturally identify certain finite strings with objects in the domain of discourse; these are usually rational numbers (elements of $\qq$) and natural numbers (elements of $\omega$). As is common, this identification takes place in the meta-theory; however, determining whether a certain string is to be identified as a rational or natural number is computable. We normally denote the string representation of such objects using an overline: if $x$ is an object in the domain of discourse, then $\str{x}$ denotes the string which we shall identify as $x$.

Fixing a particular coding is illustrative---we code objects in the domain of discourse as follows. The implied operation on finite strings below is always concatenation. We remark that these codings are not necessarily optimal.
\begin{itemize}
	\item If $k < \omega$ then let $\str{k}$ be the string whose digits are given by  the binary expansion of $k$.
	\item If $n \in \mathbb{Z}$ then let $w$ be the binary expansion of $n$ with each digit doubled ($n=101$ becomes $w=110011$). Then let $\str{n} = w01$ if $n \geq 0$, and $\str{n} = w10$ otherwise.
	\item If $q \in \qq$ then suppose $q = a/b$. Then let $\str{q} = \str{a}\str{b}$.
	\item If $q = (q_1,\ldots,q_m) \in \qq^m$ then let $\str{q} = \str{q_1}\cdots\str{q_m}$.
	\item If $x \in \rr$, suppose $k < \omega$, and express $x$ in binary. Take the integer part of $x$ and double each digit; denote this string by $w$. Take the first $k$ bits of $x$ after the binary point, denoted by $z$. If $x \geq 0$, let $\str{x}[k] = w01z$; otherwise define $\str{x}[k] = w10z$. \label{item:realDef}
	\item If $x = (x_1,\ldots,x_m) \in \rr^m$, suppose $k < \omega$. Then let $\str{x}[k] = \str{x_1}[k]\cdots\str{x_m}[k]$.
	\item If $x \in \rr$ then let $\str{x} \in \cant$ be the limit of $\str{x}[k]$ in the obvious fashion. If $x = (x_1,\ldots,x_m) \in \rr^m$ then interweave $\str{x_1},\ldots,\str{x_m}$ bit by bit.
\end{itemize}

Observe that, using this coding, if $k < \omega$ then $\len(\str{k}) \leq \log(k) + 1$.

The distinction between strings and objects is particularly important when we discuss real numbers (i.e.\ objects in $\rr$), and their truncated approximations. In other cases we are more casual; for instance, we normally write $K(k)$ and $K(q)$ instead of the formally correct $K(\str{k})$ and $K(\str{q})$.

\subsection{Dimension of points and the point-to-set principle}
Using effective tools in order to answer geometrical and measure-theoretical questions has been an avenue for research for a couple of decades. This development goes back at least to the beginning of this century \cite{lutzGales,mayordomo,lutz2003}. Among others, further discoveries were made by Hitchcock \cite{hitchcockPhD, hitchcockPaper} and Mayordomo \cite{mayordomo}, who related the notion of effective dimension of reals (in $\cant$) away from the notion of gales towards Kolmogorov complexity (the connection between martingales and Hausdorff dimension had previously been investigated by Ryabko \cite{ryabko1,ryabko2}, Staiger \cite{staiger}, and Cai and Hartmanis \cite{caihart}; gales, a generalisation of martingales, are due to Lutz \cite{lutzGalesNew}). As a result of their work, we may now determine the Hausdorff dimension of a set solely form its elements, using randomness notions.

\begin{dfn}
	Let $f \in \cant$. We define the \emph{dimension of $f$} by
	\begin{align*}
		\dim(f) &= \liminf_{r \rightarrow \infty} \frac{K(f \upharpoonright r)}{r}
	\intertext{This relativises: if $A \in \cant$ is an oracle then define}
		\dim^A(f) &= \liminf_{r \rightarrow \infty} \frac{K^A(f \upharpoonright r)}{r}.
	\end{align*}
\end{dfn}
This notion can be naturally extended to Euclidean space: first, consider the complexity of a point.

\begin{dfn}\label{dfn:efffff}
	Let $x = (x_1,\ldots,x_m) \in \rr^m$. Then we define the \emph{Kolmogorov complexity of $x$ at precision $t < \omega$} by
	\begin{align*}
		K_t(x) &= \min\set{K(q)}{q \in \qq^m \cap B_{2^{-t}}(x)}\\
		\intertext{where $B_s(y)$ is the open ball with respect to the Euclidean metric, with radius $s$ and centre $y$. The \emph{effective Hausdorff dimension} of $x$ is then given by}
		\dim(x) &= \liminf_{t \rightarrow \infty} \frac{K_t(x)}{t}.
	\end{align*}
	Both of these notions relativise.
\end{dfn}

The ideas behind the definitions stem from effectivising Hausdorff dimension, an idea that is due to Lutz \cite{lutzGalesNew} and their notion of gales. An investigation of effective packing dimension followed suit \cite{athreyaHitchcockEtAl}, also in terms of gales. The characterisation of effective Hausdorff dimension of reals given in \cref{dfn:efffff} is due to Mayordomo \cite{mayordomo}. We also remark that replacing $\liminf$ by $\limsup$ in \cref{dfn:efffff} yields a characterisation of effective packing dimension \cite{lutzMayor2}, which is denoted by $\Dim(x)$.

We can now state Lutz' and Lutz' point-to-set principle, which forms a cornerstone of our subsequent arguments.

\begin{thm}[Point-to-set Principle, {\cite[Thm.\ 1]{kakeya}}]\label{thm:pts}
	Let $n<\omega$ and $E \subset \rr^n$. Then
	\[
		\dim_H(E) = \min_{A \in \cant} \sup_{x \in E} \dim^A(x).
	\]
\end{thm}

Using this theorem, we are now in a position to control the dimension of a set by focussing on individual points. Lutz and Lutz \cite{kakeya} and Lutz and Stull \cite{effectiveDimension} provide outlines and applications of the point-to-set principle. Recently, the point-to-set principle has been extended to arbitrary separable metric spaces \cite{marsExt}.

Crucial to our arguments in this paper is the following technical lemma, which allows us to work with finite strings instead of approximating rationals. As we will be working with sequences in Cantor space $\cant$ while talking about reals in $\rr$, a convenient identification is useful. Instead of moving back and forth between $\cant$ and $\rr$, we use an observation by Lutz and Stull \cite[p.\ 6]{effectiveDimension}.

\begin{lem}[{\cite[Corollary 2.4]{lutzStullPointsonLine}}]\label{lem:stringsNreals}
	For every $m < \omega$ there exists a constant $c$ such that for all $t < \omega$ and $x \in \rr^m$ we have
	\[
		|K_t(x) - K(\str{x}[t])| \leq K(t) + c.
	\]
\end{lem}

The proof uses the fact that $x[t]$ provides a reasonable approximation to $x$, in the sense that its distance to $x$ is bounded by a function that depends on $m$. In \cref{sec:polar}, we provide a similar identification argument for polar coordinates; see \cref{prp:stringsNpolar}. We note an important corollary right here.

\begin{cor}\label{cor:dimApprox}
	If $m \geq 1$ and $x \in \rr^m$ then $\dim(x) = \liminf_{r \rightarrow \infty} \frac{K(\str{x}[r])}{r}$.
\end{cor}

\section{Arguing in polar Coordinates}\label{sec:polar}

In the course of our constructions in both \cref{thm:firstThm,thm:secondThm}, it will be easier to work in polar coordinates than in Euclidean coordinates. A point $(x,y)$ in Euclidean space \emph{has polar coordinates} $(r,\theta)$ if and only if $x = r\cos\theta$ and $y = r\sin\theta$. We will restrict our attention to the first quadrant of the unit disc, which we denote by
\[
	\mathbb{D} = \Set{(x,y) \in \rr^2}{x,y \geq 0 \land \sqrt{x^2 + y^2} \leq 1}.
\]
Thus $r \in [0,1]$ and $\theta \in [0,\pi/2]$. Importantly, all points expressed in the proofs below are given in Euclidean coordinates. When we write $(r,\theta)$ we do \emph{not} mean $(r\cos\theta,r\sin\theta)$.

The following lemma can be considered an analogue to \cref{lem:stringsNreals}; its proof follows the proof of \cref{lem:stringsNreals} from \cite[Corollary 2.4]{lutzStullPointsonLine}.

\begin{prp}\label{prp:stringsNpolar}
	Suppose $(x,y) \in \mathbb{D}$ has polar coordinates $(r,\theta)$. Then
	\[
		\dim(x,y) = \dim(r,\theta).
	\]
	This argument relativises.
\end{prp}

We provide proofs to both directions of \cref{prp:stringsNpolar} individually below. We require the following result, which is due to Casey and J.\ Lutz \cite{caseLutz}.

\begin{lem}\label{lem:LS31noConditioning}
	There exists a constant $c$ such that for all $m,s,\Delta s < \omega$ and all $x \in \rr^m$ we have
	\[
		K_s(x) \leq K_{s + \Delta s}(x) \leq K_s(x) + K(s) + c_m(\Delta s) + c
	\]
	where $c_m(\Delta s) = K(\Delta s) + m\Delta s + 2\log(\lceil \frac{1}{2}\log(m) \rceil + \Delta s + 3) + (\lceil \frac{1}{2}\log(m) \rceil + 3)m + K(m) + 2\log(m)$.
\end{lem}

Observe that the term $c_m(\Delta s)$ does not depend on $s$.

Before we commence with the proofs of the directions of \cref{prp:stringsNpolar}, we require the following lemma. This is a standard result \cite[p.\ 151]{birkhoffRota}).

\begin{lem}\label{lem:brlipschitz}
	Suppose $C \subset \rr^2$ is compact and convex. If the function $f \colon C \rightarrow \rr^2$ sending $(x,y)$ to $f(x,y)$ is continuously differentiable on $C$ then it satisfies a Lipschitz condition on $C$.
\end{lem}

The first halves of the proofs below follow the same argument as Lutz and Stull \cite[Lemma 2.3]{lutzStullPointsonLine}. We briefly observe that the map $(r,\theta) \mapsto (r\cos \theta, r\sin\theta)$ is continuously differentiable everywhere, and that $[0,1] \times [0,\pi/2]$ is of course compact and convex.

\begin{lem}[First half of \Cref{prp:stringsNpolar}]\label{lem:firstHalf}
	There exists a constant $c$ such that whenever $(x,y) \in \mathbb{D}$ has polar coordinates $(r,\theta)$ then for all $s < \omega$ we have
	\[
		K_s(x,y) \leq K(\str{r}[s]\str{\theta}[s]) + K(s) + c.
	\]
\end{lem}
\begin{proof}
By \cref{lem:brlipschitz}, the map translating polar into Cartesian coordinates satisfies a Lipschitz condition as $[0,1] \times [0,\pi/2]$ is compact and convex: there exists $M > 0$ such that if $(r,\theta),(r',\theta') \in [0,1] \times [0,\pi/2]$ then
\[
	|(r\cos\theta,r\sin\theta) - (r'\cos\theta',r'\sin\theta')| \leq M|(r,\theta) - (r',\theta')|. \tag{$*$}\label{eq:lipCond}
\]

Let $(r,\theta) \in [0,1] \times [0,\pi/2]$, and suppose $(x,y) = (r \cos \theta, r \sin \theta)$. We define the following:
\begin{itemize}
	\item let $r_s = r[s]$, and $\theta_s = \theta[s]$, the truncation of $r$ and $\theta$ to $s$ bits after the binary point;
	\item we will consider the approximation $r_s\cos\theta_s$ of $r\cos\theta$ (and similarly $r_s\sin\theta_s$ for $r\sin \theta$); however, this approximation will in general not yield a finite string; hence we define
	\[
	x[s] = (r_s\cos\theta_s)[s] \qquad \text{ and } \qquad y[s] = (r_s\sin\theta_s)[s]
	\]
	which are the truncations to $s$ bits after the binary point of such approximations.
\end{itemize}
These truncations allow us to approximate the point $(x,y)$ effectively:

\begin{claim}
	$(x[s],y[s]) \in B_{2^{-s}(1 + M\sqrt{2})}(x,y)$
\end{claim}

\begin{cproof}
Recall that $x = r\cos\theta$ and $y = r\sin\theta$, and that $x[s] = (r_s\cos\theta_s)[s]$; hence $|(x[s],y[s]) - (r_s\cos\theta_s,r_s\sin\theta_s)| \leq 2^{-s}$, by construction. Using the nomenclature introduced and the Lipschitz condition (\ref{eq:lipCond}) above we can compute the maximum error when $(x[s],y[s])$ approximates $(x,y)$:
\begin{align*}
	|(x[s],y[s]) - (x,y)| &= |(x[s],y[s]) - (r_s\cos\theta_s,r_s\sin\theta_s) + (r_s\cos\theta_s,r_s\sin\theta_s) - (x,y)|\\
	&\leq |(x[s],y[s]) - (r_s\cos\theta_s,r_s\sin\theta_s)| + |(r_s\cos\theta_s,r_s\sin\theta_s) - (x,y)|\\
	&\leq 2^{-s} + |(r_s\cos\theta_s,r_s\sin\theta_s) - (r\cos\theta,r\sin\theta)|\\
	&\leq 2^{-s} + M|(r_s,\theta_s) - (r,\theta)|\\
	&\leq 2^{-s} + M\sqrt{(r - r_s)^2 + (\theta - \theta_s)^2}\\
	&< 2^{-s} + M\sqrt{(2)2^{-2s}}\\
	&= 2^{-s} + \sqrt{2}M2^{-s}\\
	&= 2^{-s}(1 + M\sqrt{2})
\end{align*}
as required.
\end{cproof}

Observe that $2^{-t} = 2^{-s}(1 + M\sqrt{2})$ if and only if $t = s - \log(1+M\sqrt{2})$.
Therefore, if we can compute $(x[s],y[s])$ then we can compute $(x,y)$ at precision $t = s - \log(1+M\sqrt{2})$. Letting $\Delta t = \log(1 + M\sqrt{2})$ we hence have
\[
	K_{s-\Delta t}(x,y) \leq K(x[s],y[s]) \leq K(\str{x}[s]\str{y}[s]) + c'
\]
where $c'$ is the machine constant that turns the string representations $\str{x}[s]$ and $\str{y}[s]$ into the real approximations, and hence rationals, $x[s]$ and $y[s]$.

Now, $t + \Delta t = s$, so the right-hand side of \cref{lem:LS31noConditioning} implies
\begin{align*}
	K_s(x,y) &\leq K_{s - \Delta t}(x,y) + K(t) + c_2(\Delta t) + c\\
	&\leq K(\str{x}[s]\str{y}[s]) + c' + K(t) + c_2(\Delta t) + c
\end{align*}
where $c_2(\Delta t)$ is as in \cref{lem:LS31noConditioning} and hence does not depend on $t$, and thus not on $s$.


Finally, we prove the following claim:
\begin{claim}\label{claim:taylor}
	$K(\str{x}[s]\str{y}[s]) \leq K(\str{r}[s]\str{\theta}[s]) + c''$ for some constant $c''$.
\end{claim}
\begin{cproof}
	Recall that $r_s = r[s]$ and $\theta_s = \theta[s]$. Then $x[s] = (r_s\cos\theta_s)[s]$, which (as well as $y[s]$) is easily computable via approximations and Taylor's Theorem (multiplication and $\cos$ are obviously computable, with associated machine constant $c''$).
\end{cproof}


To recap, we established so far that
\begin{align*}
	K_s(x,y) &\leq K(\str{x}[s]\str{y}[s]) + c' + K(t) + c_2(\Delta t) + c.
\intertext{By the claim we now have}
	&\leq K(\str{r}[s]\str{\theta}[s]) + c' + c'' + K(t) + c_2(\Delta t) + c.
\end{align*}
Recall that $\Delta t = \log(1 + M\sqrt{2})$ which is constant. Further, $t = s - \Delta t$, and thus there exists a constant $c'''$ (which only depends on $\Delta t$, and hence only on $M$ and not on $s$) for which $K(t) = K(s - \Delta t) \leq K(s) + c'''$. Hence we have
\[
	K_s(x,y) \leq K(\str{r}[s]\str{\theta}[s]) + K(s) + d
\]
where $d = c' + c'' + c''' + c_2(\Delta t) + c$ as required.
\end{proof}

For the second half, we make the following brief observation. The argument below will focus on points in $\mathbb{D}$ that do not lie on the first axis; this is necessary for a bounding argument involving Lipschitz conditions: for each point $(x,y)$ not on the first axis, we can find a nice neighbourhood on which the coordinate transformation map from Euclidean to polar coordinates is nicely behaved. What about the points \emph{on} the first axis? There is nothing to do, for if $x \geq 0$, the polar coordinates and Euclidean coordinates of the point $(x,0)$ coincide. Hence \cref{prp:stringsNpolar} holds on the first axis trivially.

\begin{lem}[Second half of \Cref{prp:stringsNpolar}]\label{lem:secondLemma}
	There exists a constant $c$ such that whenever $(x,y) \in \mathbb{D}$ has polar coordinates $(r,\theta)$ then there exist $N_{(x,y)} < \omega$ and $\Delta < \omega$ such that if $s > N_{(x,y)}$ then
	\[
		K(\str{r}[s - \Delta]\str{\theta}[s - \Delta]) \leq K_s(x,y) + K(s) + c.
	\]
\end{lem}
\begin{proof}
	First, we make an approximating observation.
	
	\begin{claim}
	For $a \in \qq^2 \cap B_{2^{-r}}(x,y)$ we have $(x[r],y[r]) \in B_{2^{-r}(1 + \sqrt{2})}(a)$.
\end{claim}
\begin{cproof}
	By assumption, $|(x,y) - a| < 2^{-r}$, so by the triangle inequality we have
	\begin{align*}
		|(x[r],y[r]) - a| &\leq |(x[r],y[r]) - (x,y)| + |(x,y) - a|\\
		&= \sqrt{(x[r] - x)^2 + (y[r] - y)^2} + |(x,y) - a|\\
		&\leq \sqrt{2(2^{-2r})} + 2^{-r}\\
		&\leq 2^{-r}\sqrt{2} + 2^{-r}\\
		&= 2^{-r}(1 + \sqrt{2})
	\end{align*}
	as required.
\end{cproof}

Further, in the notation of Lutz and Stull \cite{lutzStullPointsonLine}, let $\mathcal{Q}_r^2 = \set{2^{-r}z}{z \in \mathbb{Z}^2}$ denote the set of $r$-dyadics. Observe that $r$-dyadics have at most $r$-many non-zero post-binary-point bits. It is easy to bound the number of $r$-dyadics in any open ball. In particular, we have:

\begin{claim}
       For any $a \in \mathbb{Q}^2$ and $r < \omega$, we have
       \[
       |\mathcal{Q}_r^2 \cap B_{2^{-r}(1 + \sqrt{2})}(a)| \leq (4(1 + \sqrt{2}))^2.
       \]
   \end{claim}
   \begin{cproof}
       Let $C_2$ be the square with side length $2(1+\sqrt{2})2^{-r}$ that is centred at $a$. It is clear that $B_{2^{-r}(1 + \sqrt{2})} \subset C_2$ and thus
       \[
           |\mathcal{Q}_r^2 \cap B_{2^{-r}(1 + \sqrt{2})}| \leq |\mathcal{Q}_r^2 \cap C_2|.
       \]
       Observe that $C_2$ has area $(2(1+\sqrt{2}))^22^{-2r}$. Now, if $x,y \in \mathcal{Q}_r^2$ and $x \neq y$ then $|x-y| \geq 2^{-r}$ (since the elements in $\mathcal{Q}_r^2$ have at most $r$-many non-zero post-binary-point bits). Hence consider a \emph{small square}: a square of side length $2^{-r}$. Such a small square has area $2^{-2r}$ and cannot contain more than $4$ $r$-dyadics: one on each of its vertices. Hence, dividing the area of $C_2$ by the area of a small square and multiplying by $4$ for each vertex gives an upper bound for the number of $r$-dyadics:
       \[
           |\mathcal{Q}^2_r \cap B_{2^{-r}(1 + \sqrt{2})}| \leq \frac{(2(1+\sqrt{2}))^2 2^{-2r}}{2^{-2r}}2^2 = (2(1+\sqrt{2}))^2 2^2 = (4(1 + \sqrt{2}))^2
       \]
       as required.
   \end{cproof}
   
	
	Let $M$ with program $P$ be a machine that does the following: on input $\pi = \pi_1\pi_2\pi_3$ if $h(\pi_1) = \str{k}$ with $k < \omega$, and $h(\pi_2) = \str{t}$ with $t < \omega$, and $h(\pi_3) = \str{a}$ with $a = (p,q) \in \qq^2$, then $M$ outputs the $k$-th dyadic rational in $B_{2^{-t}(1 + \sqrt{2})}(a)$. Suppose $a \in \qq^2$ witnesses the complexity of $K_s(x,y)$; then the claims together imply that $(x[s],y[s])$ is the $k$-th element in $\mathcal{Q}_r^2 \cap B_{2^{-r}(1 + \sqrt{2})}(a)$ for some $k < (4(1 + \sqrt{2}))^2$. Let the programs $\pi_1,\pi_2,\pi_3$ be witnesses for $K(k),K(s)$ and $K(a) = K_s(x,y)$, respectively. Then
	\[
		h\left( 0^{\len(P)}1P\pi_1\pi_2\pi_3 \right) = \str{x}[s]\str{y}[s]
	\]
	and thus
	\begin{align*}
		K(\str{x}[s]\str{y}[s]) &\leq \len(\pi_1) + \len(\pi_2) + \len(\pi_2) + c\\
		&= K(k) + K(s) + K_s(x,y) + c'\\
		&\leq K_s(x,y) + K(s) + c
	\end{align*}
	where $K(k)$ can be bounded above and hence only contributes a constant term.
	
	Let $f \colon \rr^2 \rightarrow \rr^2$ be the computable function mapping a point in Euclidean coordinates to its polar coordinates. On $\mathbb{D}$ (excluding the first axis), this map is given by $(x,y) \mapsto \left( \sqrt{x^2 + y^2}, \tan^{-1}(y/x) \right)$, and is continuously differentiable. Hence take some $\epsilon > 0$ such that the closed ball $B$ of radius $\epsilon$ centred at $(x,y)$ does not intersect the first axis. By \cref{lem:brlipschitz}, the map $f$ satisfies a Lipschitz condition on $B$. Now suppose $s < \omega$ is such that $2^{-s} < \epsilon$; thus $B_{2^{-s}}(x,y) \subset B$. Suppose $(p,q) \in \qq^2 \cap B_{2^{-s}}(x,y)$. Recalling that $f(x,y) = (r,\theta)$ we have
	\begin{align*}
		|(r,\theta) - f(p,q)| &= |f(x,y) - f(p,q)|\\
		&\leq M|(x,y) - (p,q)|\\
		&\leq M2^{-s}\\
		&= 2^{-(s - \log M)}.
	\end{align*}
	
	Thus, computing $(x[s],y[s])$ yields, after applying the machine that computes $f$ with machine constant $c''$ (compare with \cref{claim:taylor} of the proof of \cref{lem:firstHalf}), the polar coordinates $(r,\theta)$ up to precision $s - \log M$. In other words,
	\begin{align*}
		K(\str{r}[s - \log M]\str{\theta}[s - \log M]) &\leq K(\str{x}[s]\str{y}[s]) + c''\\
		&\leq K_s(x,y) + K(s) + c'' + c
	\end{align*}
	as needed.
\end{proof}

\begin{proof}[Proof of \Cref{prp:stringsNpolar}]
	The proof is an easy consequence of the previous two lemmas and the following claim, which is easily seen to be true.
	\begin{claim}
		If $\Delta < \omega$ then $|K(\str{r}[s]\str{\theta}[s]) - K(\str{r}[s - \Delta]\str{\theta}[s - \Delta])| \leq c$ for some constant $c$.
	\end{claim}
	\begin{cproof}
		It is easy to compute $\str{r}[s - \Delta]\str{\theta}[s - \Delta]$ from $\str{r}[s]\str{\theta}[s]$. For the other direction, let $\str{r}(\Delta)$ be such that $\str{r}[s] = \str{r}[s - \Delta]\str{r}(\Delta)$, and equally for $\str{\theta}$. Suppose $h(\pi_1) = \str{r}[s - \Delta]\str{\theta}[s - \Delta]$, and $h(\pi_2) = \str{r}(\Delta)$ and $h(\pi_3) = \str{\theta}(\Delta)$, and all such programs are optimal. Let $p$ be a program that on input $\pi = \pi_1\pi_2\pi_3$, merges the two strings obtained by $\pi_2$ and $\pi_3$ with the string from $\pi_1$ in the obvious way (recall the coding from \cref{sec:codings}). Then
		\[
			h\left( 0^{\len(p)}1p\pi_1\pi_2\pi_3 \right) = \str{r}[s]\str{\theta}[s]
		\]
		and thus
		\begin{align*}			
			K(\str{r}[s]\str{\theta}[s]) &\leq \len(\pi_1) + \len(\pi_2) + \len(\pi_3) + c\\
			&= K(\str{r}[s - \Delta]\str{\theta}[s - \Delta]) + K(\str{r}(\Delta)) + K(\str{\theta}(\Delta)) + c
		\intertext{by optimality. Observe that $\len(\str{r}(\Delta)) = \Delta$, and recall that $K(\sigma) \leq \len(\sigma) + 2\log(\len(\sigma)) + c'$ if $\sigma \in \fcant$; thus}
			K(\str{r}[s]\str{\theta}[s]) &\leq K(\str{r}[s - \Delta]\str{\theta}[s - \Delta]) + 2\len(\str{\Delta}) + 4\log(\len(\str{\Delta})) + c
		\end{align*}
		which, since $\Delta < \omega$ is fixed, is as required.
	\end{cproof}
	
	The claim now yields the result from the previous two lemmas: let $(x,y) \in \mathbb{D}$ with polar coordinates $(r,\theta)$. Suppose $\Delta$ is as in \cref{lem:secondLemma}. Then
	\begin{align*}
		\dim(r,\theta) &= \liminf_{s \rightarrow \infty} \frac{K_s(r,\theta)}{s}\\
		&= \liminf_{s \rightarrow \infty} \frac{K(\str{r}[s]\str{\theta}[s])}{s}\\
		&= \liminf_{s \rightarrow \infty} \frac{K(\str{r}[s - \Delta]\str{\theta}[s - \Delta])}{s}\\
		&= \liminf_{s \rightarrow \infty} \frac{K_s(x,y)}{s}\\
		&= \dim(x,y)
	\end{align*}
	using the fact that $K(s) \leq \log(s) + 2\log(\log(s) + 1) + c$ for some constant.
\end{proof}

We may now pass to polar coordinates as required. In particular, the points of the co-analytic sets we build in \cref{thm:firstThm} will be determined by their radius, which we will construct explicitly.

From now on, if we write $(r,\theta)$ below, we usually mean the point that has Euclidean coordinates $(r\cos\theta,r\sin\theta)$; in such cases, $(r,\theta) \in \mathbb{D}$. We will occasionally return to Euclidean coordinates, however, and we will explicitly mention when we do so.


\subsection{Projections in polar coordinates}
The focus of this paper is on projections of points onto straight lines. We make some simple geometric observations below that will simplify arguments later on. Consider $\theta \in [0,\pi]$, and let $L_{\theta}$ be the straight line that passes through the origin at angle $\theta$ with the first coordinate axis. It is clear that $[0,\pi]$ exhausts all straight lines through the origin. Let $(s,\rho) \in \mathbb{D}$ and denote by $\proj_{\theta}(s,\rho)$ the \emph{projection of $(s,\rho)$ onto $L_{\theta}$}: the unique point of intersection of $L_\theta$ with the unique perpendicular-to-$L_\theta$ line containing $(s,\rho)$. Recall that if $(s,\rho) \in \mathbb{D}$ then $0 \leq s \leq 1$. See \cref{fig:manyProjections}.

There are two cases: either $|\theta - \rho| \leq \pi/2$ or $|\theta - \rho| > \pi/2$. If $|\theta - \rho| \leq \pi/2$ then the length of the projection is given by $|\proj_{\theta}(s,\rho)| = s\cos(\theta - \rho)$; otherwise $|\proj_{\theta}(s,\rho)| = s\cos((\theta + \pi) - \rho)$. Since $\cos(x + \pi) = -\cos(x)$ and $0 \leq s \leq 1$, we conclude:

\begin{lem}\label{lem:polarCoords}
For every $(s,\rho) \in \mathbb{D}$ and every $\theta \in [0,\pi]$ we have
\[
	|\proj_{\theta}(s,\rho)| = s \lvert\cos(\theta - \rho)\rvert.
\]
	In particular, the polar coordinates of the projection of $(s,\rho)$ onto $L_{\theta}$ are
\begin{align*}
	\proj_{\theta}(s,\rho) =
		\begin{cases}
			(s\lvert\cos(\theta - \rho)\rvert, \theta) & \text{if $|\theta - \rho| \leq \pi/2$}\\
			(s\lvert\cos(\theta - \rho)\rvert, \theta + \pi) & \text{otherwise.}
		\end{cases}
\end{align*}
\end{lem}
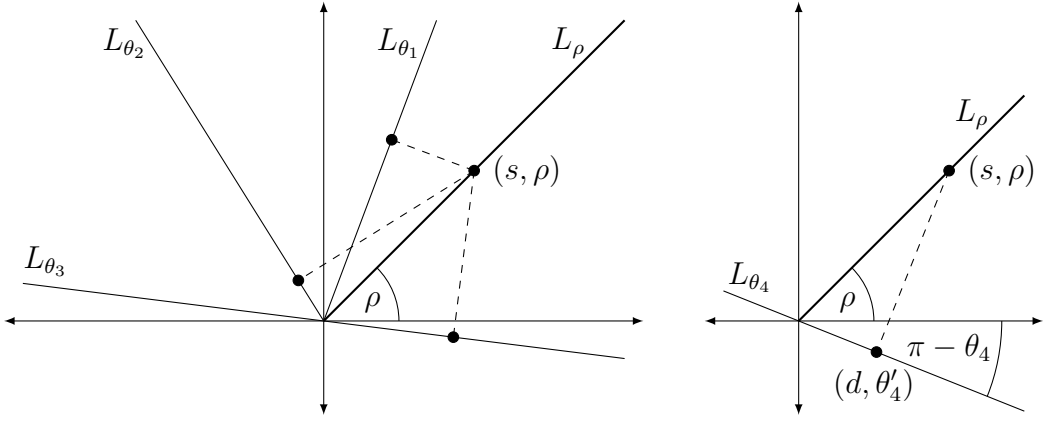
\begin{figure}%
    \centering

\begin{tikzpicture}

\draw[>=latex,<->] (-4.25,0) -- (4.25,0); 
\draw[>=latex,<->] (0,-1.25) -- (0,4.25); 

\draw[thick] (0,0) -- (45:{sqrt(32)});
\node at (3.25,3.7) {$L_{\rho}$};

\filldraw (45:{sqrt(8)}) circle (2pt);
\node at (2.7,2) {$(s,\rho)$};

\draw (0,0) -- (69.44:4.272);
\node at (1,3.7) {$L_{\theta_1}$};

\draw (0,0) -- (122:4.717);
\node at (-2.65,3.7) {$L_{\theta_2}$};

\draw (0,0) -- (-7.12502:4.03113);
\draw (0,0) -- (172.875:4.03113);
\node at (-3.7,0.8) {$L_{\theta_3}$};

\draw[dashed] (45:{sqrt(8)}) -- (69.44:2.575);
\filldraw (69.44:2.575) circle (2pt);

\draw[dashed] (45:{sqrt(8)}) -- (122:0.63626);
\filldraw (122:0.63626) circle (2pt);

\draw[dashed] (45:{sqrt(8)}) -- (-7.12502:1.736486);
\filldraw (-7.12502:1.736486) circle (2pt);

\node at (0.65,0.25) {$\rho$};
\draw (1,0) arc (0:45:1);

\end{tikzpicture}%
%
\hspace{0.8cm}%
\begin{tikzpicture}

\draw[>=latex,<->] (-1.25,0) -- (3.25,0); 
\draw[>=latex,<->] (0,-1.25) -- (0,4.25); 

\draw[thick] (0,0) -- (45:{sqrt(18)});
\node at (2.3,2.75) {$L_{\rho}$};

\filldraw (45:{sqrt(8)}) circle (2pt);
\node at (2.7,2) {$(s,\rho)$};

\draw (-1,0.4) -- (3,-1.2);
\node at (-0.65,0.58) {$L_{\theta_4}$};

\draw[dashed] (45:{sqrt(8)}) -- (-21.8014:1.11417);

\node at (0.65,0.25) {$\rho$};
\draw (1,0) arc (0:45:1);

\node at (2,-0.35) {$\pi - \theta_4$};
\draw (-21.8014:2.7) arc (-21.8014:0:2.7);

\filldraw (-21.8014:1.11417) circle (2pt);
\node at (1,-0.85) {$(d,\theta'_4)$};

\end{tikzpicture}%
%
\caption{With $(s,\rho) \in L_{\rho}$, the projections onto lines with angles $\theta_1,\theta_2$ are straightforward. For $\theta_3,\theta_4$, the projections meet in the fourth quadrant. There, $\theta'_4 = \pi - \theta_4$, and the definition of $\cos$ yields $d = s\cos(\rho + \pi - \theta_4) = s \lvert\cos(\theta_4 - \rho)\rvert$.}
\label{fig:manyProjections}
\end{figure}%
Now suppose $E \subset \mathbb{D}$ and fix some $\theta \in [0,\pi]$. Define
\[
	E(\theta) = \set{s\lvert\cos(\theta - \rho)\rvert}{(s,\rho) \in E} \subset \rr.
\]
We show below that, in fact, $\dim_H(E(\theta)) = \dim_H(\proj_{\theta}(E))$.

We need the following notions: a real number $x \in \rr$ is \emph{computable} if there exists a machine that uniformly on input $k < \omega$ (or rather $\str{k}$) outputs a rational $q \in \qq$ (or rather $\str{q}$) such that $q \in B_{2^{-k}}(x)$; this naturally extends to $\rr^m$ for $m \geq 1$.

\begin{lem}\label{lem:computableRealZeroDim}
Let $m\geq 1$. Every computable real $x \in \rr^m$ has dimension $0$.
\end{lem}

\begin{proof}
	Suppose $M$ with program $p$ is a machine that on input $\str{s}$ for $s < \omega$ computes $\str{q_s}$ for some $q_s \in \qq^m \cap B_{2^{-s}}(x)$. Then $h\left( 0^{\len(p)}1p\str{s} \right) = \str{q_s} \in \qq^m \cap B_{2^{-s}}(x)$ and so $K_s(x) \leq \len(\str{s}) + c$. Recall that $\len(\str{s}) \leq \log(s) + 1$ and thus
	\[
		\dim(x) \leq \liminf_{s \rightarrow \infty} \frac{\log(s) + 1 + c}{s} = 0
	\]
	as needed.
\end{proof}

\begin{lem}\label{lem:countDimZero}
Every countable set $E \subset \rr^2$ has Hausdorff dimension $0$.
\end{lem}
\begin{proof}
Suppose $E = \set{x_i}{i < \omega}$, and let $X = \bigoplus \str{x_i}$, the infinite join. Let $M$ with program $p$ be a machine with oracle to access to $X$ that on input $(\str{i},\str{s})$ computes $\str{x_i}[s]$. Then it is clear that $M$ computes all $x_i$, and hence by \cref{lem:computableRealZeroDim} and the point-to-set principle \cref{thm:pts} we have
\[
	\dim_H(E) \leq \sup_{x \in E} \dim^X(x) = 0
\]
as needed.
\end{proof}

\begin{lem}\label{lem:oneTwoDim}
	Let $r \in \rr$. Then for every oracle $A \in \cant$ the following hold.
	\begin{enumerate}
		\item $\dim^A(r) = \dim^A(r,0)$ \label{item:pairDim}
		\item $\dim^A(r) = \dim^A(-r)$ \label{item:negDim}
	\end{enumerate}
\end{lem}
\begin{proof}
	It is easily seen, modulo machine constants, that
	\[
		K(\str{r}[s]) \leq K(\str{r}[s]\str{0}[s]) \leq K(\str{r}[s]) + K(\str{0}[s]).
	\]
	Since $0$ is computable, \cref{lem:computableRealZeroDim} implies $\lim_{s \rightarrow \infty} \frac{K(\str{0}[s])}{s} = 0$. Applying $\liminf$ yields \cref{item:pairDim}. For \cref{item:negDim}, observe that it is easy to compute $\str{-r}[s]$ from $\str{r}[s]$, from which the result follows immediately. Both arguments relativise.
\end{proof}

\begin{lem}\label{lem:projPth}
	Let $\theta \in [0,\pi)$. If $E \subset \mathbb{D}$ then
	\begin{align*}
		\dim_H(\proj_\theta(E)) = \dim_H(E(\theta)).
	\end{align*}
\end{lem}
\begin{proof}
	Fix $\theta \in [0,\pi]$ and suppose $(s,\rho) \in \mathbb{D}$. For brevity, define $p(s,\rho)$ so that
	\[
		p(s,\rho) = |\proj_{\theta}(s,\rho)| = s\lvert\cos(\theta - \rho)\rvert
	\]
	by \cref{lem:polarCoords}. Now \cref{item:pairDim} of \cref{lem:oneTwoDim} implies
	\[
		\dim^A(p(s,\rho)) = \dim^A(p(s,\rho),0)
	\]
	for every oracle $A \in \cant$. Hence let
	\[
		P_{\theta}(E) = \set{(p(s,\rho),0)}{(s,\rho) \in E} \subset \rr^2.
	\]
	It is now easy to see that $\dim_H(E(\theta)) = \dim_H(P_{\theta}(E))$ by the point-to-set principle \cref{thm:pts}.
	
	We now aim to appeal to \cref{cor:rotationPreservesDimH}: Hausdorff dimension is invariant under rotations. However, rotating $P_{\theta}(E)$ by $\theta$ anti-clockwise is not necessarily equal to $\proj_{\theta}(E)$: if there exists $(s,\rho) \in E$ for which $|\theta - \rho| > \pi/2$ then $\proj_{\theta}(s,\rho) = (p(s,\rho),\theta + \pi)$, not $(p(s,\rho),\theta)$. This is easily accounted for: whenever $(s,\rho) \in E$ and $|\theta - \rho| > \pi/2$ then, passing to Euclidean coordinates, consider $(-p(s,\rho),0)$ instead. To this end, let
	\begin{align*}
		p^*(s,\rho) &= \begin{cases}
			p(s,\rho) &\text{if $|\theta - \rho| \leq \pi/2$}\\
			-p(s,\rho) &\text{otherwise.}
		\end{cases}
	\intertext{and hence define, in Euclidean coordinates, the set}
		P_{\theta}^*(E) &= \set{(p^*(s,\rho),0)}{(s,\rho) \in E};
	\end{align*}
	see \cref{fig:pThetaProjections}. By \cref{item:pairDim,item:negDim} of \cref{lem:oneTwoDim}, it is immediate that $\dim^A(p(s,\rho),0) = \dim^A(p^*(s,\rho),0)$ for all oracles $A \in \cant$. Hence the point-to-set principle implies that $\dim_H(P_{\theta}(E)) = \dim_H(P_{\theta}^*(E))$. Further, rotating $P_{\theta}^*(E)$ by $\theta$ yields $\proj_{\theta}(E)$. Hence applying \cref{cor:rotationPreservesDimH} shows
	\[
		\dim_H(E(\theta)) = \dim_H(P_{\theta}(E)) = \dim_H(P^*_{\theta}(E)) = \dim_H(\proj_{\theta}(E))
	\]
as required.
\end{proof}

\begin{figure}%
    \centering
\begin{tikzpicture}

\draw[>=latex,<->] (-1.25,0) -- (3.25,0); 
\draw[>=latex,<->] (0,-1.25) -- (0,4.25); 

\draw[thick] (0,0) -- (60:4.618802);
\node at (1.7,3.65) {$L_{\theta_1}$};

\draw (0,0) -- (75:4.141105);
\node at (0.57,3.65) {$L_{\rho_2}$};

\draw (0,0) -- (20:3.363102);
\node at (2.85,1.45) {$L_{\rho_1}$};

\filldraw (20:1.5) circle (2pt); 
\filldraw (75:2.85) circle (2pt);

\filldraw (60:1.299038) circle (2pt);
\draw[dashed] (20:1.5) -- (60:1.299038);
\filldraw (60:2.752889) circle (2pt);
\draw[dashed] (75:2.85) -- (60:2.752889);

\draw (0:1.299038) arc (0:60:1.299038);
\draw (0:2.752889) arc (0:60:2.752889);

\filldraw[draw=black,fill=white] (0:1.299038) circle (2pt);
\filldraw[draw=black,fill=white] (0:2.752889) circle (2pt);

\node at (1.4,1.25) {$(r_1,\theta_1)$};
\node at (2.15,2.45) {$(r_2,\theta_1)$};

\draw [
    decorate, 
    decoration = {brace,mirror}]
        (0.1,-0.1) -- (1.2,-0.1)
	node[pos=0.5,below=0.05,black]{$r_1$};
\draw [
    decorate, 
    decoration = {brace,mirror}]
        (0.1,-0.65) -- (2.65,-0.65)
	node[pos=0.5,below=0.05,black]{$r_2$};

\end{tikzpicture}%
%
\hspace{0.8cm}%
\begin{tikzpicture}

\draw[>=latex,<->] (-4.25,0) -- (4.25,0); 
\draw[>=latex,<->] (0,-1.25) -- (0,4.25); 

\draw[thick] (0,0) -- (150:4.618802);
\draw[thick] (0,0) -- (-30:2);
\node at (-3.7,2.5) {$L_{\theta_2}$};

\draw (0,0) -- (30:4.618802);
\node at (3.6,2.5) {$L_{\rho_3}$};

\draw (0,0) -- (80:4.061706);
\node at (1,3.45) {$L_{\rho_4}$};

\filldraw (30:3.25) circle (2pt); 
\filldraw (80:3.25) circle (2pt); 

\filldraw (-30:1.625) circle (2pt); 
\filldraw (150:1.111565) circle (2pt); 

\draw[dashed] (30:3.25) -- (-30:1.625);
\draw[dashed] (80:3.25) -- (150:1.111565);

\draw (0:1.111565) arc (0:150:1.111565);
\draw (180:1.625) arc (180:200:1.625);
\draw[dashed] (200:1.625) arc (200:225:1.625);
\draw[dashed] (315:1.625) arc (315:320:1.625);
\draw (320:1.625) arc (320:330:1.625);

\filldraw[draw=black,fill=white] (180:1.625) circle (2pt);
\filldraw[draw=black,fill=white] (0:1.111565) circle (2pt);

\draw [
    decorate, 
    decoration = {brace,mirror}]
        (-1.525,-0.1) -- (-0.1,-0.1)
	node[pos=0.5,below=0.05,black]{$r_3$};
	
\node at (2.7,-0.6) {$(r_3,\theta_2 + \pi)$};
\node at (-1.25,1.25) {$(r_4,\theta_2)$};

\end{tikzpicture}%
%
\caption{If $|\rho - \theta| \leq \pi/2$ then it suffices to consider the length of the projections on the first axis, and rotate (see $\rho_1,\rho_2$ and $\theta_1$). Otherwise, we need to mirror along the second axis and then rotate by $\theta$; see $\theta_2$ and $\rho_3$.}
\label{fig:pThetaProjections}
\end{figure}
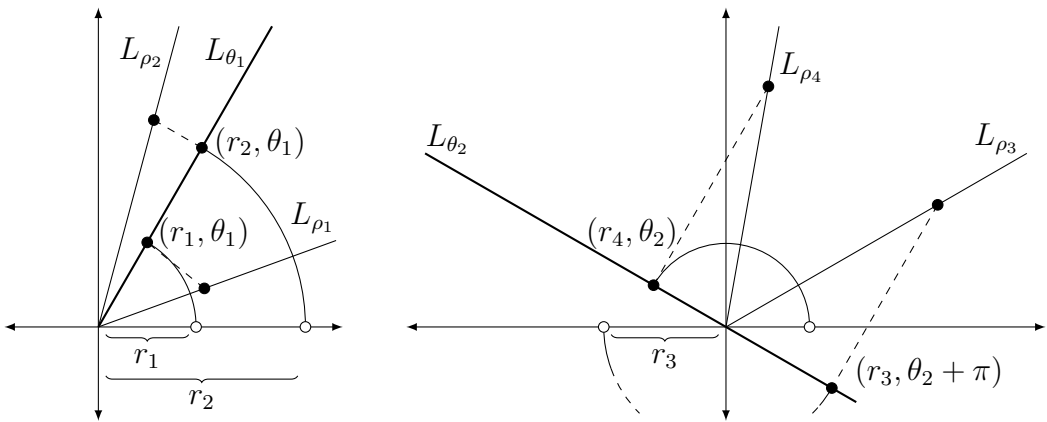%


\section{Constructing co-analytic sets by recursion}\label{sec:zoltan}
Z.\ Vidny\'{a}nszky \cite{zoltan} has developed a general recursive construction of co-analytic sets using co-analytic conditions, assuming $\VL$. While their result is very general (it uses descriptive set theoretical tools and applies to all Polish spaces), we only focus on a special case, which we describe below. The result is a generalisation of a method used by A.\ Miller \cite{miller}. The earliest version of the method is due to Erd\H{o}s, Kunen, and Mauldin \cite{erdosEtAl}.

For notational simplicity, if $X = \{ x_{\alpha} \mid \alpha < \omega_1 \}$ we define $X \upharpoonright \alpha = \{ x_{\beta} \mid \beta < \alpha \}$.

\begin{dfn}\label{dfn:compatibleF}
Given $F \subset \mathbb{D}^{\leq \omega} \times [0,\pi/2] \times \mathbb{D}$, a set $X = \set{x_{\alpha}}{\alpha < \omega_1}$ is \emph{compatible with $F$} if the following exist:
\begin{itemize}
	\item an enumeration $\set{p_{\alpha}}{\alpha < \omega_1}$ of $[0,\pi/2]$; and
	\item an enumeration $\set{A_{\alpha}}{\alpha < \omega_1} \subset \mathbb{D}^{\leq\omega}$ such that if $\alpha < \omega_1$ then $A_{\alpha} = X \upharpoonright \alpha$
\end{itemize}
such that for each $\alpha < \omega_1$ we have $(A_{\alpha},p_{\alpha},x_{\alpha}) \in F$.
\end{dfn}

Observe that since $A_{\alpha} \in \mathbb{D}^{\leq\omega}$ each $A_{\alpha}$ has order type $\leq \omega$. This is well-defined since $\omega_1$ is the least uncountable ordinal, hence for every $\omega \leq \beta < \omega_1$ there is a bijection between $\beta$ and $\omega$, providing the bounded ordering of order type $\omega$. 

\begin{dfn}\label{dfn:turingCofinal}
	A set $X \subset \cant$ is \emph{cofinal in the Turing degrees} if it is cofinal in the partial ordering of Turing degrees. If $m \geq 1$ and $X \subset \rr^m$ then $X$ is cofinal in the Turing degrees if the set $\set{\str{x}}{x \in X}$ is.
\end{dfn}

\begin{thm}[{\cite[Thm.\ 1.3]{zoltan}}, $\VL$]\label{thm:zoltansThm}
Let $F \subset \mathbb{D}^{\leq \omega} \times [0,\pi/2] \times \mathbb{D}$. If $F$ is co-analytic and if for all $(A,p) \in \mathbb{D}^{\leq \omega} \times [0,\pi/2]$ the section
\begin{align*}
     F(A,p) = \{ x \in \mathbb{D} \mid F(A,p,x) \}                                                                                                                                                                                                  \end{align*}
is cofinal in the Turing degrees, then there exists a co-analytic set $X \subset \mathbb{D}$ that is compatible with $F$.
\end{thm}

It should be noted that the theorem above has been proven to hold for all Polish spaces and all uncountable Borel subsets of an arbitrary Polish space \cite{zoltan}.

\Cref{thm:zoltansThm} proves a type of recursion principle: if $[0,\pi/2] = \set{p_{\alpha}}{\alpha < \omega_1}$ is considered the set of \emph{conditions}, then $\mathbb{D}$ is the set of \emph{candidates}. The theorem guarantees the existence of a set $X = \{ x_{\alpha} \mid \alpha < \omega_1 \}$ which, a posteriori, can be considered as being constructed in stages: at each stage a particular condition is satisfied (by picking a suitable candidate) without violating already satisfied conditions.\\

We outline the intuition: at stage $\alpha$ we have access to $A_{\alpha}$ (the set of elements we have already enumerated into $X$) and to the current condition to be satisfied, $p_{\alpha}$. The section $F(A_{\alpha},p_{\alpha})$ now gives the set of \emph{suitable candidates} which both satisfy condition $p_{\alpha}$ and respect $A_{\alpha}$. Since $A_{\alpha} = X \upharpoonright \alpha$, we see that $X$ satisfies all conditions and is coherent (in a sense similar to that of coherence in inverse limits in category theory).

In our particular case, we think of $\mathbb{D}$ as the set of candidates, and of $[0,\pi/2]$ as the set of conditions: these are the angles we project on. We satisfy conditions differently in \cref{thm:firstThm,thm:secondThm}: let $\theta \in [0,\pi/2]$ be the current condition to be satisfied. In \cref{thm:firstThm}, we pick a suitable candidate on the straight line through the origin at angle $\theta$; in \cref{thm:secondThm}, we use $\theta$ as an oracle to find a point (lying on a different line) that is sufficiently complicated relative to it.

We then show that the sections containing suitable candidates are cofinal in the Turing degrees. We prove that for every $X \in \cant$ there exists $x \in F(A,p)$ such that $x$ codes some $Y \geq_T X$ via some $m$-reduction; we incorporate those reductions via what we call \emph{folding maps} in \cref{sec:provingCofinality}. (For an introduction to computability theoretic notions such as $m$-reductions see Soare \cite{soare}.)\\

Vidny\'{a}nszky and Medini have provided applications of \cref{thm:zoltansThm} \cite{zoltan,mediniVid}: among others, they construct a co-analytic two-point set as well as a co-analytic Hamel basis and MAD (maximally almost disjoint) family of sets. All of these results had been obtained previously by A.\ Miller \cite{miller}.


\section{The Proof of \texorpdfstring{\Cref{thm:firstThm}}{our first Theorem}}\label{sec:firstThm}

In this section, we construct the following counterexample: a plane set of Hausdorff dimension $1$, all of whose projections have dimension $0$. Using the results from \cref{sec:polar}, we will argue in polar coordinates.

\begin{thm}[$\VL$]\label{thm:firstThm}
	There exists a co-analytic set $E \subset \rr^2$ such that $\dim_H(E) = 1$ while, for every $\theta \in [0,\pi]$ we have $\dim_H(\proj_{\theta}(E)) = 0$.
\end{thm}

We make use of the following classical theorem. We give a proof using effective dimension, and the point-to-set principle \cref{thm:pts}.

\begin{lem}\label{lem:meetsEveryLine}
	If $E \subset \rr^2 \setminus \{ 0 \}$ intersects every line through the origin in $\mathbb{D}$, then $\dim_H(E) \geq 1$.
\end{lem}

\begin{proof}
	Let $A \in \cant$ be an oracle. There exists $B \in \cant$ random relative to $A$. Thus $\str{\theta} = 0001B \in \cant$ codes a real $\theta \in (0,1)$. Since $B$ is random relative to $A$, we know $K^A(B \upharpoonright s) \geq s-c$ for some constant $c$. As $B \upharpoonright s$ is easily computable from $\str{\theta}[s]$ we have
	\[
		s-c \leq K^A(B \upharpoonright s) \leq K^A(\str{\theta}[s]) + c'
	\]
	for some machine constant $c'$. Thus
	\[
		\dim^A(\theta) = \liminf_{s \rightarrow \infty} \frac{K^A(\str{\theta}[s])}{s} \geq \liminf_{s \rightarrow \infty} \frac{K^A(B \upharpoonright s)}{s} \geq \liminf_{s \rightarrow \infty} \frac{s-c}{s} = 1.
	\]
	Since $E$ intersects the line with angle $\theta$, there exists $r > 0$ such that $(r,\theta) \in E$. Therefore
	\[
		\dim^A(r,\theta) = \liminf_{s \rightarrow \infty} \frac{K^A(\str{r}[s]\str{\theta}[s])}{s}.
	\]
	We can easily compute $\str{\theta}[s]$ from $\str{r}[s]\str{\theta}[s]$, so $K^A(\str{\theta}[s]) \leq K^A(\str{r}[s]\str{\theta}[s]) + c''$ for some machine constant $c''$. Hence
	\begin{align*}
		\dim^A(r,\theta) = \liminf_{s \rightarrow \infty} \frac{K^A(\str{r}[s]\str{\theta}[s])}{s} \geq \liminf_{s \rightarrow \infty} \frac{K^A(\str{\theta}[s])}{s} = \dim^A(\theta) = 1.
	\end{align*}
	Since $A$ was arbitrary, the result follows.
\end{proof}

\subsection{The roadmap towards a proof}\label{sec:roadmap}
We assume $\VL$, and hence let $B = \set{\theta_{\alpha}}{\alpha < \omega_1}$ be an enumeration of $[0,\pi/2]$. We want to argue by induction on $\omega_1$ and hence build $E \subset \mathbb{D}$ satisfying \cref{thm:firstThm} in stages; we think of the angles in $B$ as the \emph{conditions} (or \emph{requirements}) which need to be satisfied. During our construction, when considering condition $\varphi$, we also handle $\varphi + \pi/2$ at the same time. By \cref{thm:zoltansThm}, at stage $\alpha$ we have access to all points $(r_i,\theta_i)$ already enumerated into $E$. We aim to satisfy condition $\theta_{\alpha}$. As a shorthand, denote $\theta = \theta_{\alpha}$. We argue as follows:
\begin{enumerate}[label=(\arabic*)]
	\item Let $A_{\alpha} = \set{(r_i,\theta_i)}{i < \omega}$, the set of points already enumerated into $E$. For each $i < \omega$ the angular coordinate $\theta_i$ tells us which condition we have already satisfied.
	\item Construct $r \in (0,1)$ such that $\dim(r\lvert\cos(\theta - \theta_i)\rvert) = 0$ and $\dim(r\lvert\cos(\theta + \pi/2 - \theta_i)\rvert) = 0$ for all $i < \omega$. This suffices by \cref{lem:projPth}.\label{item:realExists}
	\item Enumerate the pair $(r,\theta)$ into $E$.
\end{enumerate}

Observe that the set of reals in \cref{item:realExists} must be cofinal in the Turing degrees for \cref{thm:zoltansThm} to apply. The following proposition is essential.

\begin{prp}\label{prp:turingCofinality}
	Suppose $a_i \in (0,1)$ for all $i < \omega$. There exists $r \in (0,1)$ such that $\dim(a_ir) = 0$ for all $i < \omega$. The set of such $r$ is cofinal in the Turing degrees.
\end{prp}

We will postpone the proof of \cref{prp:turingCofinality} to \cref{sec:provingCofinality}. However, having it in hand we may already give a proof of \cref{thm:firstThm}. One additional lemma is needed before we do so.

\begin{lem}\label{lem:dimIsBorelFunc}
	For every $A \in \cant$ and $a \in \mathbb{R}$ the set $\set{x \in \rr}{\dim^A(x) = a}$ is Borel.
\end{lem}
\begin{proof}
We need to show that $\dim$ as a function of $x$ is Borel measurable. Recall its definition
\[
	\dim(x) = \liminf_{n \rightarrow \infty} \frac{K_n(x)}{n} = \liminf_{n \rightarrow \infty} \frac{\min \set{K(q)}{q \in \qq \cap B_{2^{-n}}(x)}}{n}.
\]
Since the $\liminf$ of a sequence of Borel measurable functions is itself Borel measurable, it suffices to show that $K_n(x) = \min \set{K(q)}{q \in \qq \cap B_{2^{-n}}(x)}$ is Borel measurable. This is easily seen: observe that
\begin{align*}
	K_n(x) < c &\iff \exists q \in \qq \cap B_{2^{-n}}(x) (K(q) < c)\\
	&\iff x \in \bigcup_{q \in \mathcal{K}(c)} B_{2^{-n}}(q)
\end{align*}
where $\mathcal{K}(c) = \set{p \in \qq}{K(p) < c}$. Hence $K_n$ is Borel measurable, and thus so is $\dim$. This argument relativises.
\end{proof}

\begin{proof}[Proof of \Cref{thm:firstThm}]
	We use \cref{thm:zoltansThm} and define $F \subset \mathbb{D}^{\leq \omega} \times [0,\pi/2] \times \mathbb{D}$ such that
\begin{center}
$(A,\varphi,(r,\theta)) \in F$ if and only if\\
$\varphi = \theta$ and for all $(r',\theta') \in \ran(A)$ we have
\begin{align*}
	\dim(r\lvert\cos(\varphi - \theta')\rvert) = \dim(r\lvert\cos(\varphi + \pi/2 - \theta')\rvert) = 0.
\end{align*}
\end{center}
In particular, observe that every point witnessing that condition $\varphi$ is satisfied lies on the line $L_{\varphi}$.

In order to apply \cref{thm:zoltansThm}, we must show that $F$ is co-analytic; but this follows immediately from \cref{lem:dimIsBorelFunc}. Hence let $\varphi \in [0,\pi/2]$. We now focus on the sections of $F$: by definition, given $\alpha < \omega_1$ we have
\[
	F(A,\varphi) = \set{(r,\theta)}{(A,\varphi,(r,\theta)) \in F}.
\]
Suppose $A = \set{(r_i,\theta_i)}{i < \omega} \in \mathbb{D}^{\leq\omega}$, and hence countable. Let
\begin{align*}
	a_i = \lvert\cos(\varphi - \theta_i)\rvert \text{ and } b_i = \left\lvert\cos\left(\varphi + \pi/2 - \theta_i\right)\right\rvert.
\end{align*}
Observe that, by construction, we have $(r,\theta) \in F(A,\varphi)$ if and only if $\theta = \varphi$ and $\dim(ra_i) = \dim(rb_i) = 0$ for all $i < \omega$. Now \cref{prp:turingCofinality} implies that this section is cofinal in the Turing degrees. Therefore, using \cref{lem:dimIsBorelFunc}, we see that \cref{thm:zoltansThm} is applicable: there exists a co-analytic set
\[
	E = \set{(r_{\alpha},\theta_{\alpha})}{\alpha < \omega_1} \subset \rr^2
\]
which is compatible with $F$. In particular, there exist enumerations $\set{\varphi_{\alpha}}{\alpha < \omega_1} = [0,\pi/2]$ and $\set{A_{\alpha}}{\alpha < \omega_1}$ of $A_{\alpha} = \set{(r_{i}, \theta_{i})}{i < \omega} = E \upharpoonright \alpha$ such that for each $\alpha < \omega_1$,
\[
	(r_{\alpha},\theta_{\alpha}) \in F(A_{\alpha},\varphi_{\alpha}).
\]
In particular, $\theta_{\alpha} = \varphi_{\alpha}$.

For the verification, let $\varphi \in [0,\pi]$. We show that $\dim_H(\proj_{\varphi}(E)) = 0$. By \cref{lem:projPth}, it suffices to show that $\dim_H(E(\varphi)) = 0$, where $E(\varphi) = \set{r \lvert\cos(\varphi - \theta)\rvert}{(r,\theta) \in E}$. This is what we show below.

Firstly, observe that either $\varphi = \varphi_{\delta} \in [0,\pi/2]$ for some $\delta < \omega_1$; or $\varphi = \varphi_{\delta} + \pi/2 \in (\pi/2,\pi]$ for some $\varphi_{\delta} \in [0,\pi/2]$. Let $\delta$ be such, and recall that $E = \set{(r_{\alpha},\theta_{\alpha})}{\alpha < \omega_1}$. We consider the points that were enumerated before condition $\varphi_\delta$ and those enumerated after $\varphi_\delta$ separately.

\begin{itemize}
	\item[$\leq \delta$:] At condition $\varphi_\delta$, define (analogous to \cref{lem:countDimZero}) the oracle
	\[
		X = \bigoplus \Set{\str{r_{\beta}\lvert\cos(\varphi_{\delta} - \theta_{\beta})\rvert},\str{r_{\beta}\lvert\cos(\varphi_{\delta} + \pi/2 - \theta_{\beta})\rvert}}{\beta \leq \delta}.
	\]
	Then $X$ computes $r_{\beta}\lvert\cos(\varphi_{\delta} - \theta_{\beta})\rvert$ and $r_{\beta}\lvert\cos(\varphi_{\delta} + \pi/2 - \theta_{\beta})\rvert$ for all $\beta \leq \delta$. Since either $\varphi = \varphi_{\delta}$ or $\varphi = \varphi_{\delta} + \pi/2$, \cref{lem:computableRealZeroDim} implies in particular that
	\[
		\dim^X(r_{\beta}\lvert\cos(\varphi - \theta_{\beta})\rvert) = 0
	\]
	for all $\beta \leq \delta$.
	\item[$> \delta$:] We show that for every $\beta > \delta$ we have $\dim(r_{\beta}\lvert\cos(\varphi - \theta_{\beta})\rvert) = 0$. Let $\delta < \beta < \omega_1$. Then $(r_{\beta},\theta_{\beta}) \in F(A_{\beta},\varphi_{\beta}) = F(E \upharpoonright \beta, \varphi_{\beta})$. But the conditions we have already attended to at stage $\beta$ are exactly the angular coordinates of the points enumerated into $E \upharpoonright \beta$; in particular, $E \upharpoonright \beta = \set{(r_{\alpha},\varphi_{\alpha})}{\alpha < \beta}$. So for all $\gamma < \beta$, again by definition of $F$, we have
	\[
		\dim(r_{\beta}\lvert\cos(\varphi_{\gamma} - \theta_{\beta})\rvert) = \dim(r_{\beta}\lvert\cos(\varphi_{\gamma} + \pi/2 - \theta_{\beta})\rvert) = 0.
	\]
	Since $\delta < \beta$ and either $\varphi = \varphi_{\delta}$ or $\varphi = \varphi_{\delta} + \pi/2$ we have in particular
	\[
		\dim(r_{\beta}\lvert\cos(\varphi - \theta_{\beta})\rvert) = 0.
	\]
	We picked $\delta < \beta < \omega_1$ arbitrarily, hence this holds for all such $\beta$, as required.
\end{itemize}
Thus, by the point-to set principle \cref{thm:pts} and \cref{lem:projPth}, we have
\begin{align*}
\dim_H(\proj_{\varphi}(E)) &= \dim(E(\varphi))\\
	 &= \min_{A \in 2^{\omega}} \sup_{\alpha < \omega_1} \dim^A(r_{\alpha}\lvert\cos(\varphi - \theta_{\alpha})\rvert)\\
	&\leq \sup_{\alpha < \omega_1} \dim^X(r_{\alpha}\lvert\cos(\varphi - \theta_{\alpha})\rvert)\\
	&= 0.
\end{align*}
Now $\dim_H(E) \geq 1$ by \cref{lem:meetsEveryLine}.
\end{proof}

Finally, the fact that $\dim_H(E) = 1$ is a consequence of the following corollary.

\begin{cor}\label{cor:optimalDim}
	Suppose $E \subset \mathbb{D}$. Then $\dim_H(\proj_{\theta}(E)) \geq \dim_H(E) - 1$.
\end{cor}

\begin{proof}
	Suppose $(r,\theta) \in \proj_{\theta}(E)$. By \cref{lem:polarCoords}, we know that $r = s \lvert \cos(\theta - \rho) \rvert$ for some $(s,\rho) \in E$. But this means there is only one piece of information missing: from $(r,\theta)$, we can compute $s$ from $\rho$, and vice versa. Hence suppose $\dim(r,\theta) = \epsilon$. Since $\dim(s),\dim(\rho) \leq 1$ we see that $\dim(s,\rho) \leq \dim(r,\theta) + 1$, which is as desired.
\end{proof}

\subsection{Proving \texorpdfstring{\Cref{prp:turingCofinality}}{the Turing cofinality proposition}}\label{sec:provingCofinality}

An interval is \emph{(open) dyadic} if it is of the form $(j/2^k,(j+1)/2^k)$. Intervals of the form $[j/2^k,(j+1)/2^k]$ are \emph{closed dyadic}. Observe that if $x \in (j/2^k,(j+1)/2^k)$ then $|x - j/2^k| \leq 2^{-k}$, and hence $x$ and $j/2^k$ agree on the first $k$ bits in their binary expansion: both start with the binary expansion of $j$.

In the results below, we work with open intervals in $(0,1)$. All reals are expressed in binary. Instead of manipulating intervals directly, we will argue in terms of dyadic reals, which we will express by their finite binary expansion. For this, we introduce the following notation. If $\sigma \in 2^{\leq\omega}$, let $\real{\sigma} = 0.\sigma \in \rr$. If $\sigma \in \fcant$, let $\realo{\sigma} = 0.\sigma1^{\infty} \in \rr$; let $[\real{\sigma}]$ denote the open interval $(\real{\sigma},\realo{\sigma})$. If $a \in \rr$ then $a[\real{\sigma}] = (a\real{\sigma},a\realo{\sigma})$.

Some basic facts that follow directly from our definitions are:
\begin{itemize}
	\item If $\sigma \in \fcant$ then $[\real{\sigma}]$ is a dyadic interval; so if $x \in [\real{\sigma}]$ then $x$ and $\real{\sigma}$ agree on the initial segment of length $\len(\sigma)$. We can think of \emph{$x$ extending $\real{\sigma}$}.
	\item Conversely, if $I$ is dyadic and $\sigma \in \fcant$ is such that $\real{\sigma}$ is the left-end point of $I$, then $I = [\real{\sigma}]$.
	\item If $I$ is dyadic and $\sigma \in \fcant$ is such that $\real{\sigma} \in I$ then $[\real{\sigma}] \subset I$.
	\item In particular, if $\sigma,\rho \in \fcant$ then $\sigma \prec \rho$ if and only if $\real{\rho} \in [\real{\sigma}]$.
\end{itemize}

\begin{lem}\label{lem:extExist}
	Let $\sigma \in \fcant$ and $a \in (0,1)$. Suppose $0 < \epsilon < 1$. There exist $\rho,\tau \in \fcant$ such that:
	\begin{enumerate}
		\item $\sigma \prec \rho$
		\item $a[\real{\rho}] \subset [\real{\tau}]$
		\item $K(\tau)/\len(\tau) < \epsilon$
	\end{enumerate}
\end{lem}

\begin{proof}
	Let $\sigma,a$ and $\epsilon$ be given. Consider $a[\real{\sigma}]$. Since $[\real{\sigma}]$ is open, so is $a[\real{\sigma}]$, and thus it contains a closed dyadic interval. Take the largest (in diameter) such interval $I$, and pick $\tau' \in \fcant$ such that $\real{\tau}'$ is the left end-point of $I$. By closedness, $\real{\tau}' \in a[\real{\sigma}]$. By standard results on Kolmogorov complexity, there exists a smallest $s < \omega$ such that $\tau = \tau'0^s$ satisfies
	\[
		\frac{K(\tau)}{\len(\tau)} < \epsilon.
	\]
	In particular, $\real{\tau}' = \real{\tau} \in I$. Now consider $[\real{\tau}]$, which is open and hence so is $a^{-1}[\real{\tau}]$. Let $J$ denote the largest closed dyadic interval contained in $a^{-1}[\real{\tau}]$, and call its left end-point $d$. Again by closedness, $d \in a^{-1}[\real{\tau}]$. Let $\rho \in \fcant$ be such that $\real{\rho} = d$.
	
	\begin{figure}
		\includegraphics{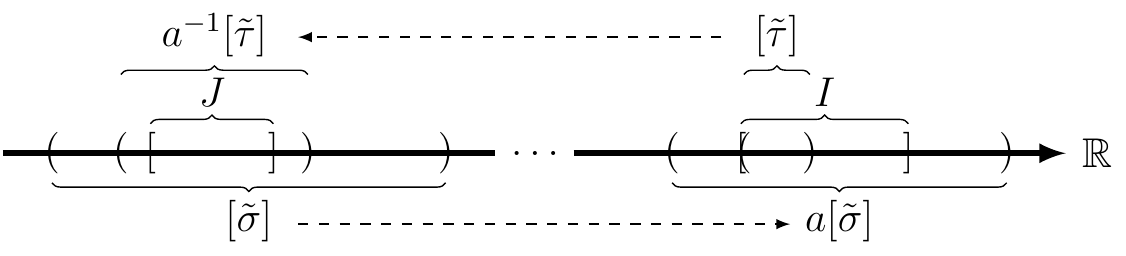}
		\caption{We start on the left and argue anti-clockwise: considering $a[\real{\sigma}]$ yields an open interval; the largest closed dyadic interval inside is $I$. Picking a suitable $\real{\tau} \in I$ yields $a^{-1}[\real{\tau}]$. The largest dyadic interval contained in it is $J$ with left end-point $d = \real{\rho}$. Hence $[\real{\rho}] \subset J$, where in fact the interior of $J$ equals $[\real{\rho}]$.}\label{fig:miamiFig}
	\end{figure}
	
	Now $\sigma \prec \rho$: by construction, $\real{\rho} = d \in J \subset a^{-1}[\real{\tau}]$. The string $\tau$ properly extends $\tau'$, thus $[\real{\tau}] \subset [\real{\tau}']$. Since $\real{\tau}'$ is the left end-point of $I$, the interior of $I$ equals $[\tilde{\tau}']$. Hence $[\real{\tau}] \subset [\real{\tau}'] \subset I \subset a[\real{\sigma}]$, and so $\real{\rho} \in a^{-1}[\real{\tau}] \subset a^{-1}(a[\real{\sigma}]) = [\real{\sigma}]$ as needed. Further, $a[\real{\rho}] \subset [\real{\tau}]$, since $[\real{\rho}] \subset J \subset a^{-1}[\real{\tau}]$ with all inclusions proper. This completes the argument.
\end{proof}

In order to achieve cofinality in the Turing degrees when constructing a suitable $r \in (0,1)$, we need to satisfy each condition (as per \cref{item:realExists} in \cref{sec:roadmap}) while coding a given oracle $A \in \cant$ into $r$. Let
\[
	\nu(k) = 2^{2^k}
\]
determine at which bits of $r$ to code $A$. We will use the gaps in between the range of $\nu$ to satisfy the conditions. We call $\nu$ the \emph{folding map}.

\subsubsection{The construction of $r$}\label{par:construction}
Suppose $(a_i)$ is the set of conditions, where $a_i \in (0,1)$ for all $i < \omega$. We construct $r \in (0,1)$ in stages, by determining its binary expansion, which is given by successive extensions $x_0 \prec x_1 \prec x_2 \prec \ldots$ with $x_i \in \fcant$. We argue by induction on $\omega$.
\begin{enumerate}[label=(\arabic*)]
	\item Let $A \in \cant$ be given.
	\item Let $x_0 = \emptyset$, the empty string.
	\item Let $x_k$ be given. At stage $k + 1$, decode $k + 1 = \langle i,n \rangle$ via Cantor's pairing function, for instance, and attend to requirement $i$. Hence we attend to each requirement infinitely often.
	\item Apply \cref{lem:extExist} with $a = a_i$ and $\epsilon = \frac{1}{k}$ to obtain a suitable extension $\rho_k \succ x_k$.
	\item Let $t = \nu(k+1) - \len(\rho_k) - 1$ and $d = A(k)$ and define
	\[
		x_{k+1} = \begin{cases}
			\rho_k0^td & \text{if $\len(\rho_k) < \nu(k+1)$}\\
			(\rho_k \upharpoonright (\nu(k+1) - 1))d & \text{otherwise.}
		\end{cases}\label{eq:defR}
	\]
	Therefore, if $k > 0$ then $\len(x_k) = \nu(k)$ by induction. \label{item:lengthXk}
	\item Define $x = \bigcup_{k < \omega} x_k$, and hence let $r = \real{x}$.
	\item Observe that $A$ is computably \emph{folded into $x$}: for all $k < \omega$, we have $x(\nu(k+1)-1) = A(k)$.
\end{enumerate}

In order to complete the proof of \cref{prp:turingCofinality}, we need to ensure that the second case in the equation in \cref{eq:defR} only occurs finitely often for each requirement $a_i$. The following lemma assures us that this is indeed the case.

Before we proceed with the proof, another couple of useful facts about intervals follow. Let $(x,y) \subset (0,1)$.
\begin{enumerate}[label=(\roman*)]
	\item By $\diam((x,y)) = y - x$ we denote the diameter of $(x,y)$. If $\sigma \in \fcant$ then $\diam([\real{\sigma}]) = 2^{-\len(\sigma)}$. In particular, $-\log(\diam([\real{\sigma}])) = \len(\sigma)$. \label{item:logLength1}
	\item If $k < \omega$ is such that $k \geq -\log(\diam((x,y))) + 2$ then there exists $j < \omega$ such that the closed dyadic interval $[j/2^k,(j+1)/2^k] \subset (x,y)$. \label{item:logLength2}
\end{enumerate}

\begin{lem}\label{lem:eventuallySpace}
	For each $a_i \in (0,1)$ there exists $M_i < \omega$ such that if $k + 1 > M_i$ and $k + 1 = \langle i,n \rangle$ attends to requirement $a_i$, then $\len(\rho_k) < \nu(k + 1)$.
\end{lem}

\begin{proof}
	Fix some $a_i = a$ and suppose we are at stage $k + 1 = \langle i,n \rangle$. Let $\rho = \rho_k$. Recall that $\real{\rho} \in J \subset a^{-1}[\real{\tau}]$. Observe that $\diam(a^{-1}[\real{\tau}]) = a^{-1}2^{-\len(\tau)}$. Now, since $J$ is defined to be the maximal (in diameter) closed dyadic interval inside $a^{-1}[\real{\tau}]$, and since $\real{\rho}$ is the left end-point of $J$, \cref{item:logLength1,item:logLength2} imply
	\begin{align*}
		\len(\rho) &\leq -\log(\diam(a^{-1}[\real{\tau}])) + 2\\
		&= \log(a) - \log\left(2^{-\len(\tau)}\right) + 2 = \log(a) + \len(\tau) + 2.
		\intertext{Recall that $\tau =\tau'0^s$, and hence}
		\len(\rho) &\leq \log(a) + \len(\tau') + s + 2.
		\intertext{Recall that $\rho$ is an extension of $x_k$ (so $x_k = \sigma$ in \cref{lem:extExist}). By construction, $\real{\tau}' \in I \subset a[\real{x}_k]$, where $I$ is dyadic maximal in $a[\real{x}_k]$. Therefore}
		\len(\tau') &\leq -\log(\diam(a[\real{x}_k])) + 2 = -\log(a) + \len(x_k) + 2
		\intertext{from which we obtain via \cref{item:lengthXk} that}
		\len(\rho) &\leq \len(x_k) + s + 4 = \nu(k) + s + 4
	\end{align*}
	Consider some stage $k+1$. We are building $x_{k+1} \succ x_k$, where $\len(x_{k+1}) = \nu(k+1)$. Our construction is successful if we need not truncate $\rho$ (as in the latter case in \cref{eq:defR}). In such a case, $\len(\rho) < \len(x_{k+1}) = \nu(k+1)$. Hence it suffices to show that, eventually, $s < \nu(k+1) - \nu(k) - 4$.
	
	Recall that $s$ is chosen so that $\frac{K(\tau)}{\len(\tau)} = \frac{K(\tau'0^s)}{\len(\tau') + s} < \frac{1}{k}$. Simplify this as follows:
	\begin{align*}
		\frac{K(\tau'0^s)}{\len(\tau') + s} &\leq \frac{K(\tau') + K(0^s) + c'}{s}\\
		&\leq \frac{K(\tau')}{s} + \frac{K(s)}{s} + \frac{c''}{s}\\
		&\leq \frac{\len(\tau') + 2\log(\len(\tau'))}{s} + \frac{\log(s) + 2\log(\log(s) + 1)}{s} + \frac{c}{s}
	\end{align*}
	for a sum of machine constants $c$.
	
	These terms are easily bounded. Clearly, $\frac{c}{s} < \frac{1}{3k}$ if $s > 3kc$. For the middle term, observe that if $s \geq 2$ then $\log(s) + 2\log(\log(s) + 1) < 3\log(s)$. Hence,
	\[
		\frac{\log(s) + 2\log(\log(s) + 1)}{s} < \frac{3\log(s)}{s}.
	\]
	Since $\log(s)/s$ is monotonically decreasing, if $s > 2^k$ then 
	\[
		\frac{3\log(s)}{s} < \frac{3\log\left(2^k\right)}{2^k} = \frac{3k}{2^k}.
	\]
	Then $\frac{3k}{2^k} < \frac{1}{3k}$ if $9k^2 < 2^k$ which holds for $k \geq 10$. Hence, for large enough $k$, the bound $s > 2^k$ suffices.
	
	For the first term, recall that $\len(\tau') \leq -\log(a) + \nu(k) + 2$. Since $a \in (0,1)$ we know $-\log(a) > 0$. So, for large enough $k$, it follows that
	\begin{align*}
		\frac{\len(\tau') + 2\log(\len(\tau'))}{s} &\leq \frac{-\log(a) + \nu(k) + 2 + 2\log(-\log(a) + \nu(k) + 2)}{s}\\
		&\leq \frac{-\log(a) + 3\nu(k)}{s}
		\end{align*}
	Since $a$ is fixed we have, for large enough $k$, that
	\begin{align*}
		\frac{\len(\tau') + 2\log(\len(\tau'))}{s} \leq \frac{-\log(a) + 3\nu(k)}{s} \leq \frac{4\nu(k)}{s}.
	\end{align*}
	Now observe that $\frac{4\nu(k)}{s} \leq \frac{1}{3k}$ if $s > 12k\nu(k)$. Choosing a large enough $k$ we hence see that $s > \max \left\{ 3kc, 2^k, 12k\nu(k) \right\}$ suffices, which, again, reduces to $s > 12k\nu(k)$ once $k$ is large enough.
	
	Finally, note that $12k\nu(k) + 1 < \nu(k+1) - \nu(k) - 4$ for $k \geq 3$. Thus, once $k$ is sufficiently large to satisfy all conditions above, $s = 12k\nu(k) + 1$ satisfies $\frac{K(\tau'0^s)}{s} < \frac{1}{k}$ while $s < \nu(k+1) - \nu(k) - 4$. So, eventually, $\len(\rho)$ is small enough, as required.
\end{proof}

\begin{proof}[Proof of \Cref{prp:turingCofinality}]
	Let $A \in 2^{\omega}$ be given, and suppose $(a_i)$ is the countable sequence of requirements. Construct $x = \bigcup_{k < \omega} x_k$ as in \cref{par:construction}. Let $r = \real{x}$. From \cref{sec:codings} and \cref{dfn:turingCofinal} it is easily seen that $A$ can be obtained computably from the binary expansion of $r$. Hence we only need to show that the dimension of $a_ir$ is minimal. Fix $i < \omega$ and consider $a_i$. By \cref{lem:eventuallySpace} there exists $M$ such that if $k > M$ and $k = \langle i,n \rangle$ then $\rho_k \prec x$. For each such $k$, let $\tau_k$ be as obtained from \cref{lem:extExist} alongside $\rho_k$. Now
	\begin{align*}
		\dim(a_ir) &= \liminf_{s \rightarrow \infty} \frac{K(\str{a_ir}[s])}{s}
		\intertext{by \cref{cor:dimApprox}. By construction, $a_i[\real{\rho}_k] \subset [\real{\tau}_k]$ and $a_i\real{\rho}_k \in [\real{\tau}_k]$. Thus $a_ir = a_i\real{x} \in [\real{\tau}_k]$. Further, $K(\tau_k)/\len(\tau_k) < 1/k$. Let $D = \set{k > M}{ k = \langle i,n \rangle \text{ for some $n$}}$. Then}
		\dim(a_ir) &\leq \liminf_{k \rightarrow \infty,\, k \in D} \frac{K(\str{a_ir}[\len(\tau_k)])}{\len(\tau_k)}\\
		&\leq \liminf_{k \rightarrow \infty,\, k \in D} \frac{K(\tau_k) + c}{\len(\tau_k)}\\
		&\leq \liminf_{k \rightarrow \infty,\, k \in D} \frac{1}{k}\\
		&= 0
	\end{align*}
	where $c$ is the machine constant obtaining $\tau_k$ from $\str{a_ir}[\len(\tau_k)]$ (as per \cref{sec:codings}). This completes the proof.
\end{proof}


\section{The Proof of \texorpdfstring{\Cref{thm:secondThm}}{our second Theorem}}\label{sec:secondThm}
In this section provide a proof to the following result, which is optimal by \cref{cor:optimalDim}.

\begin{thm}\label{thm:secondThm}
	For every $0 < \epsilon < 1$, there exists a co-analytic set $E \subset \rr^2$ such that $\dim_H(E) = 1 + \epsilon$ while, for every $\theta \in [0,2\pi)$ we have $\dim_H(\proj_{\theta}(E)) = \epsilon$.
\end{thm}

Observe that this is in fact a generalisation of \cref{thm:firstThm}: the case $\epsilon=0$ is covered there. Secondly, the case $\epsilon = 1$ is trivial: if $E \subset \rr^2$ satisfies $\dim_H(E) = 2$ then \cref{cor:optimalDim} implies $1 \leq \dim_H(\proj_\theta(E)) \leq 1$ for each $\theta$. Hence our theorems exhaust all cases.

\subsection{Roadmap towards a proof.} Let $0 < \epsilon < 1$. Assuming $\VL$, we argue as follows.
\begin{enumerate}
	\item Fix an enumeration $\set{\varphi_{\alpha}}{\alpha < \omega_1}$ of $[0,\pi/2]$.
	\item At stage $\alpha$, let $A_{\alpha} = \set{(r_i,\theta_i)}{i < \omega}$, the set of all points already enumerated into our set.
	\item Let $X \in \cant$ be the sequence whose bits are made up of the binary expansion of $\varphi_{\alpha}$. In particular, $X$ is $\str{\varphi}_{\alpha}$ with its first four bits removed (cf.\ \cref{sec:codings}).
	\item We will \emph{not} satisfy condition $\varphi_{\alpha}$ by enumerating a point on $L_{\varphi_{\alpha}}$ into our set. Instead, we recover the already satisfied conditions by first coding them into $r$ using a suitable folding map: if $(r_i,\theta_i)$ was enumerated into our set at stage $\beta$, then $\varphi_{\beta}$ is folded into $r_i$, and can hence be recovered computably. Let $\set{\varphi_i}{i < \omega}$ be the set of the conditions already satisfied.
	\item Pick $\theta \in [0,\pi/2]$ such that $\str{\theta}$ is random relative to $X$.
	\item Let $(a_i)$ be an enumeration of all $\lvert\cos(\theta-\varphi_i)\rvert$ and $\lvert\cos(\theta +\pi/2 -\varphi_i)\rvert$. Let $Y$ be the join of $X$ and all $\str{a}_i$. We also assume that $Y$ can compute $\epsilon$.
	\item Construct $r \in (0,1)$ such that:
	\begin{enumerate}
		\item the binary expansion of $\varphi_\alpha$ is folded computably into the binary expansion of $r$; \label{item:foldPhi}
		\item $\dim(ra_i) = \epsilon$ for all $i < \omega$; \label{item:controlProj}
		\item $\dim^{Y,\str{\theta}}(r) = \epsilon$. \label{item:controlRad}
	\end{enumerate}
\end{enumerate}

\begin{rem}
	At first, it appears difficult how to control \cref{item:controlProj,item:controlRad}. In practice, construct $r$ so that $\dim(ra_i) \leq \epsilon$ and $\dim^{Y,\str{\theta}}(r) \geq \epsilon$. Equality then follows immediately from \cref{cor:optimalDim}.
\end{rem}

We will give some insight into the verification below. Let $(a_i)$ be an enumeration of all $\lvert\cos(\theta - \varphi_i)\rvert$ and $\lvert\cos(\theta + \pi/2 - \varphi_i)\rvert$.

In our construction of a suitable $r$, we adapt the methods used in the proof of \cref{thm:firstThm}. However, instead of inserting long strings of zeroes into the binary expansions of $ra_i$, we pick a suitable oracle $T \in \cant$ and fold it into $ra_i$.

The oracle $T$ is suitable if it is random relative to $Y \oplus \str{\theta}$ (and hence all $\str{a}_i$). Now suppose $r$ is as constructed. Then $Y$ (which computes all $\str{a}_i$) can compute an initial segment of $T$ from an initial segment of $r$: just compute an initial segment of $ra_i$ for the correct $i$. Since $T$ is random relative to $Y \oplus \str{\theta}$, we can force $\dim^{Y,\str{\theta}}(r)$ to not dip too low by coding $T$ not too sparsely.

The details can be found in \cref{sec:verificationThm2}, and the theorem then follows by \cref{cor:optimalDim}. Further, we use the following simple result, which follows from symmetry of information.

\begin{lem}\label{lem:supAdditive}
	Let $A \in \cant$ be an oracle. For any $x,y \in \rr$ we have $\dim^A(x,y) \geq \dim^A(x) + \dim^{A,\str{x}}(y)$.
\end{lem}

For an in-depth account of the interplay between relativised dimension and \emph{conditional dimension} of elements of $\rr^n$ see Lutz and Lutz \cite[4.3, 4.4]{kakeya}, who introduced the latter notion ibidem. The previous lemma is also a consequence of their arguments, and a proof can be found there, too; it follows from the fact that $K(x \mid y) \geq K^{\str{y}}(x)$. 

Recall that $Y$ computes $X$, and hence $\dim^X(r) \geq \dim^Y(r)$ for all $r$. In particular, we now see that
\[
	\dim^X(r,\theta) \geq \dim^X(\theta) + \dim^{X,\str{\theta}}(r) \geq \dim^X(\theta) + \dim^{Y,\str{\theta}}(r) \geq 1 + \epsilon
\]
since $\str{\theta}$ is random relative to $X$, and by our construction of $r$. The final steps of this high-level verification are then as follows: suppose we construct $E$ broadly as in the proof of \cref{thm:firstThm}. By the same argument as in said proof, using \cref{thm:pts} we see that
\[
	\dim_H(\proj_{\theta}(E)) = \dim_H(E(\theta)) \leq \epsilon
\]
since allowing oracles can only decrease the dimension of points. On the other hand, every oracle $X$ appears through some $\varphi_{\alpha} \in [0,\pi/2]$. Hence there exists a point $(r_{\alpha},\theta_{\alpha})$ for which $\str{\theta}$ is random relative to $X$. Since such a point exists for \emph{every} oracle, \cref{thm:pts} implies
\[
	\dim_H(E) \geq \dim^X(r_{\alpha},\theta_{\alpha}) \geq \dim^X(\theta) + \dim^{X,\str{\theta}}(r) \geq 1 + \epsilon
\]
by \cref{lem:supAdditive}. Then the conclusion follows from \cref{cor:optimalDim}.

We flesh out the details in the subsequent sections.

\subsection{Folding a suitable oracle into $r$.} Fix $\epsilon \in (0,1)$, and let
\[
	Z = Y \oplus \str{\theta}
\]
recalling that $Y$ already contains all $a_i$.
Instead of constructing a $T \in \cant$ random relative to $Z$ and then coding it sparsely to obtain dimension $\epsilon$, we use a result first proved by Athreya, Hitchcock, Lutz, and Mayordomo \cite[Thm.\ 6.5]{athreyaHitchcockEtAl}: for every $0 \leq \alpha \leq 1$ there exists $x \in \rr$ such that $\dim(x) = \Dim(x) = \alpha$, which they obtain precisely by sparsely coding a random sequence, interleaved with strings of zeroes. Their result relativises by choosing a sequence random to the oracle we desire. Even more so, their construction shows that we  have $\dim^Z(\str{x}) = \dim(\str{x}) = \epsilon$, so letting $T = \str{x}$ for a suitable $x \in \rr$ is as required.\\

In the construction of \cref{thm:firstThm} we focused on satisfying requirements: we demanded a particular number of consecutive zeroes to appear in the image in order to push the complexity down sufficiently far; and in our verification, we showed that, eventually, the gap between conditions will be large enough so that enough zeroes (read, a sufficiently large $s$) can be accommodated. In the present argument, we need to be more careful as we must at all times be able to give a good bound on how many bits of $T$ can be computed. Hence we fix the number of bits to be appended so that there is no ``overspill''. The following lemma yields such a bound.

\begin{lem}\label{lem:thm2notTooLong}
In the argument of \cref{lem:extExist}, if $s = \nu(k+1) - \nu(k) - 5$ then $\len(\rho_k) < \nu(k+1)$.
\end{lem}

\begin{proof}
This follows from the proof of \cref{lem:eventuallySpace}: with $a, \rho, \tau'$ as in said argument, we have
\begin{align*}
	\len(\rho_k) &\leq \log(a) + \len(\tau') + s + 2 \leq \nu(k) + s + 4.
\end{align*}
Equating the term to $\nu(k+1)$ and demanding strict inequalities yields the result.
\end{proof}

In particular, the following corollary shows that, if we have space for $s$ bits to encode, we can code $s-5$ bits into the image.

\begin{cor}\label{cor:thm2notTooLongGeneral}
In the argument of \cref{lem:extExist}, with $a \in (0,1)$: if $\len(\rho) = m$ and $n > m$ then if $s = n - m - 5$ we have that $\len(\rho') < n$, where $\len(\rho')$ is the extension of $\rho$ that codes $s$ bits into $r\real{\rho}'$.
\end{cor}

We choose the folding map
\[
	\nu(k) = 2^{2^k} + k.
\]
We introduce the shift summand $k$ so as to make sure that the gaps between $\nu(k)$ and $\nu(k+1)$ have length $2^{2^{k+1}} + k + 1 - 2^{2^k} - k = 2^{2^{k+1}} - 2^{2^k} + 1$; the last bit is reserved to code a bit of $\varphi_\alpha$ into $ra_i$ (as per \cref{item:foldPhi}). Now, the gap we have available to extend is exactly of length
\begin{align}
	\nu(k+1) - \nu(k) - 1 = 2^{2^{k+1}} - 2^{2^k} = 2^{2^k}\left( 2^{2^k} - 1 \right) \label{eq:gapSizeForm}
\end{align}
which is divisible by $2^{(2^k - k)}$. This fact will be useful in the following subsection.

\subsection{Coding and saving blocks}
Na\"{i}vely, our argument should work as follows: at each stage, we construct a radius $r$ that, together with a suitable angle $\theta$, satisfies the requirement at hand. In order to preserve the high dimension of $r$ (relative to $Z$) and of the points $ra_i$, we code segments of $T$ into each $ra_i$. In particular: if $a_j$ is attended to right after $a_i$, and the last bit of $T$ coded into $ra_i$ is $T(k)$ for some $k < \omega$, then the first bit of $T$ coded into $ra_j$ at that stage is $T(k+1)$. Hence, with a long enough initial segment of $r$, the oracle $Z$ can compute a long initial segment of $T$ by just picking the correct $a_i$ (which $Z$ computes), computing $ra_i$, and picking out the coded bits of $T$.

For the sake of exposition, suppose $T \in \cant$ and consider
\[
	T^{(m,n)} = \langle T(m), T(m+1), \ldots, T(n-1) \rangle.
\]
In particular, observe that $\len(T^{(m,n)}) = n-m$, and that $T(n)$ does not appear in $T^{(m,n)}$. Now, the dimension of $ra_i$ is bounded above by the dimension of $T$: taking a sufficiently long initial segment $\str{ra_i}[t]$ of $ra_i$, we easily find a long string of the form $T^{(m,n)}$ coded into it. Provided that $\len(T^{(m,n)}) = n-m$ is large enough compared to $t$, this will force the dimension down---this latter condition is easily ensured by choosing a sparse enough folding map.

However, it is now difficult to show that the dimension of $ra_i$ does not drop properly below the dimension of $T$. The problem is that it is in general hard to tell how many bits in the multiplication of reals are determined by a single bit: e.g.\ if $a = 1/\pi$ and $\real{\sigma} = 0.\sigma$ for some $\sigma \in \fcant$, and $\tau \succ \sigma$, there is no bound on how many bits of the product $a\real{\tau}$ are correct in the sense that every extension yields the same initial segment.\footnote{For an extreme case, one considers $a$ and $\sigma$ so that $a\real{\sigma} = 1/2 - \epsilon$ for some very small $\epsilon$, and $\diam(a[\real{\sigma}]) > 1/4$. Then each bit of precision added to $\sigma$ shifts the interval a little bit to the right, and halves it. So the number of bits it takes until the first bit is determined, i.e.\ until $1/2$ does not appear in the interval, depends on $\epsilon$.}

We circumvent this issue as follows: as we extend $r$, we \emph{save blocks} of bits that are coded into $ra_i$ throughout the stage. We do this by pulling back the interval, as seen in \cref{lem:extExist}. Hence we define the \emph{block map} $\mu \colon \omega \rightarrow \omega$ by
\[
	\mu(k) = 2^{(2^k - k)}.
\]

Recall that our folding map is given by $\nu(k) = 2^{2^k} + k$. Hence, at stage $k$ with $r_k$ at hand, we have $\nu(k+1) - \nu(k)$ many bits to extend $r_k$. In particular, the \emph{number of blocks fitting into the gap of stage $k+1$} is given by
\begin{align}
	\xi(k) = \frac{\nu(k+1) - \nu(k) - 1}{\mu(k)} = \frac{2^{2^k}\left( 2^{2^k} - 1 \right)}{2^{(2^k - k)}} = 2^k \left( 2^{2^k} - 1 \right). \label{eq:howManyBlocks}
\end{align}
Note that, by our choices, we hence have $\xi(k)\mu(k) = \nu(k+1) - \nu(k) - 1$.

Now, a few lemmas are needed. Firstly, we need to have a good bound on how many bits we can code into $ra_i$ at each stage $k$, and in each block. And secondly, it is not clear that saving blocks does not cost too many bits. The first is not an issue due to \cref{cor:thm2notTooLongGeneral}. We resolve the second later in the \emph{cost lemma} \ref{lem:atMostSeven}, after introducing the construction in detail.

As we code $T$ in blocks, it is prudent to describe a suitable partitioning of $T$ beforehand. We do this here: by recursion reconstruct $T$ into segments $T^j_k$, where $k$ denotes the last completed stage (so if we see $T_k^j$ then we are in stage $k+1$), and $j$ the active block. In summary, at stage $k+1$:
\begin{itemize}
	\item we code $\xi(k) = 2^k \left( 2^{2^k} - 1 \right)$-many blocks, which follows from \cref{eq:howManyBlocks};
	\item and each block of $T$ coded into the image has length $\mu(k) - 5$, as we lose $5$ bits each time as per \cref{cor:thm2notTooLongGeneral}.
\end{itemize}
Hence we obtain
\[
	T = \bigcup_{3 \leq k < \omega} \left( \bigcup_{1 \leq j \leq \xi(k)} T^j_k \right)
\]
where the union operator denotes concatenation. Hence
\[
	T = T_3^1 \cup T_3^2 \cup \ldots \cup T_3^{2040} \cup T_3^1 \cup \ldots T_{k}^{\xi(k)} \cup T_{k+1}^1 \ldots 
\]
since $\xi(3) = 2040$. As mentioned, we lose $5$ bits each time we code a block of $T$, hence
\[
	\len(T^j_k) = \mu(k) - 5 = 2^{(2^k - k)} - 5.
\]
This confirms why the outer union starts at $k = 3$: below, we have $2^{(2^2 - 2)} - 5 < 0$, so there is no space to code any bits. In particular, the first stage at which bits are coded is stage $k+1 = 4$, with $\len(T^j_k) = \len(T_3^j) = 27$ and $\xi(k) = \xi(3) = 2040$.

\subsection{The construction}
Recall that our folding and block map are $\nu(k) = 2^{2^k} + k$ and $\mu(k) = 2^{(2^k - k)}$, respectively. Now, the radius $r$ is constructed as follows: suppose $\varphi_{\alpha}$ is the active requirement.
\begin{enumerate}
	\item Let $A \in \cant$.
	\item Let $x_0 = \emptyset$, the empty string.
	\item Let $x_k$ be given. At stage $k+1$, decode $k+1 = \langle i,n \rangle$; we now attend to requirement $i$.
	\item We iterate over all $\xi(k)$-many blocks. Let $0 \leq j < \xi(k) = 2^k (2^{2^k} - 1)$.
	\begin{enumerate}
		\item Let $x_k^0 = x_k$.
		\item At block $j+1$, suppose we have $x_k^{j}$. We apply \cref{lem:extExist}, but instead of coding zeroes, we code $T^{j+1}_k$ into $a_i\real{x}_k^j$. Let $\rho_k^{j+1}$ be the resulting extension. By filling up with $s$-many zeroes (courtesy of \cref{lem:thm2notTooLong}), we hence find $x_k^{j+1} = \rho_k^{j+1}0^s$ of length
		\[
			\len(x_k^{j}) + \mu(k) = \len(x_k) + 2^{(2^k - k)}(j+1).
		\]
	\end{enumerate}
	\item After the last block, we have one bit left to code $A$ or $\str{\varphi}_{\alpha}$ (this follows from \cref{eq:gapSizeForm}). By construction, $\len\left(x_k^{\xi(k)}\right) = \nu(k+1) -1$; hence define
	\[
		x_{k+1} = x_k^{\xi(k)}d
	\]
	where
	\begin{align*}
		d = \begin{cases}
			A(k/2) & \text{if $k$ is even}\\
			\str{\varphi}_{\alpha}((k-1)/2) & \text{if $k$ is odd;}
		\end{cases}
	\end{align*}
	hence $\len(x_{k+1}) = \nu(k+1)$, as intended. Further, we code the active line into the real we are building, so that we can recover it later.
\end{enumerate}

Of course, we code $A$ as in \cref{thm:firstThm} in order to be able to apply \cref{thm:zoltansThm}. This completes the construction.

\subsection{The verification}\label{sec:verificationThm2}
In the present context, we have two results to prove: that $\dim(ra_i) \leq \epsilon$ and that $\dim^Z(r) \geq \epsilon$, where $Z = Y \oplus \str{\theta}$. Then the theorem follows from \cref{cor:optimalDim}. We prove both results individually.

\subsubsection{The dimension of $ra_i$} Both verification arguments are ``bit counting'' arguments: we exhibit a piece of a complicated string coded inside $ra_i$, and show that said segment is long enough in a precise sense: its length dwarves the length of all non-coded bits. Let $a = a_i$.

Consider $\str{ar}[m]$ for some $m$ such that
\[
	\str{ar}[m] = \sigma \cup \left( \bigcup_{1 \leq j \leq \xi(k)} \sigma_j T^j_k \right)
\]
for some $k$; hence stage $k + 1$ has just been completed.\footnote{Considering the strings at the end of stages is prudent as we easily have access to a long consecutive segment of $T$ (albeit interrupted).} We also know that $\len(\sigma) \leq -\log(a) + \len(x_k) + 2 = -\log(a) + \nu(k) + 2$, by \cref{lem:eventuallySpace}. Further, the \emph{cost} of saving a block is given by a bound on the length of each $\sigma_j$ which we give here:

\begin{lem}[The cost lemma]\label{lem:atMostSeven}
	Let $a \in (0,1)$ and $r_m \in \fcant$. As in \cref{lem:extExist}, find $\real{\tau}_m$ and $I_m$ dyadic such that $[\real{\tau}_m] \subset I_m \subset a[\real{r}_m]$; let $\tau'_m$ be the left end-point of $I_m$. Further, let $J \subset a^{-1}[\real{\tau}_m]$ be dyadic, where $\real{\rho}_k$ is the left-endpoint of $J$. Let $r_{m+1} = \rho_m0^t$ so that $\len(r_{m+1}) = \len(r_m) + \mu(k)$ where $k$ denotes the current stage. Finally, suppose $\tau'_{m+1}$ is the left end-point of $I_{m+1} \subset a[\real{r}_{m+1}]$.
	
	Then $|\len(\tau'_{m+1}) - \len(\tau_m)| \leq 7$.
\end{lem}

Before we proceed with the proof, a few comments are in order. Firstly, consulting \cref{fig:miamiFig} alongside the statement and proof of the above lemma is useful, as the figure serves as its motivation. Conceptually, one can think of the hypotheses of this lemma as the intermediate step between moving from one block to the next within a given stage in our construction: $r_m$ is the available string in block $m$ inside some stage, and $\rho \succ \sigma$ is its computed extension. Importantly, $a[\real{r}_{m+1}]$ contains $\tau_m$ as a substring. We are asking: after saving $\tau_m$ in $a[\real{r}_{m+1}]$, how many bits are lost before we begin coding the next block? In particular, if we construct a real $r$ by such approximations $r_m$ and we have established that
\[
	ar \succ \tau_m \lambda \tau_{m+1}
\]
for some $\lambda \in \fcant$ by successive block saving, then how long can $\lambda$ be at most?

\begin{proof}
	By assumption we have $[\real{\tau}_m] \subset I_m \subset [\real{r}_m]$, and so $\diam([\real{\tau}_m]) \leq \diam(I_m) \leq \diam([\real{r}_m])$. Applying $-\log$ and by \cref{item:logLength1} we have $-\log(\diam(I_m)) \in [-\log(a) + \len(r_k), \len(\tau_k)]$. Since $\real{\tau}'_k$ is the left end-point of $I_k$ we have in particular that $\len(\tau'_m) \in [-\log(a) + \len(r_m), \len(\tau_m)]$. We can give an even better bound: by \cref{item:logLength2}, we see that $\len(\tau'_m) \leq -\log(\diam(a[\real{r}_m])) + 2 = -\log(a) + \len(r_m) + 2$, and hence
	\[
		\len(\tau'_m) \in [-\log(a) + \len(r_m), -\log(a) + \len(r_m) + 2].
	\]
	
	By construction, at stage $k+1$ we code $\mu(k) - 5$ bits into the image for each block (we lose $5$ bits each block, as per \cref{cor:thm2notTooLongGeneral}). Hence $\len(\tau_m) = \len(\tau'_m) + (\mu(k) - 5)$. Therefore, observing by construction that $\len(\tau'_{m+1}) \geq \len(\tau_m)$, we see that
	\begin{align*}
		\len(\tau'_{m+1}) - \len(\tau_m) &= \len(\tau'_{m+1}) - \len(\tau'_m) - (\mu(k) - 5)\\
		&\leq -\log(a) + \len(r_{m+1}) + 2 + \log(a) - \len(r_m) - (\mu(k) - 5)\\
		&= (\len(r_{m+1}) - \len(r_m)) - (\mu(k) - 5) + 2\\
		&= \mu(k) - (\mu(k) - 5) + 2\\
		&= 7
	\end{align*}
	where we use that the block size is $\mu(k)$, and hence $\len(r_{m+1}) - \len(r_m) = \mu(k)$.
\end{proof}

Hence $\len(\sigma_j) \leq 7$. For simplicity, we let
\[
	T_k = T^1_k \cup \ldots \cup T^{\xi(k)}_k;
\]
hence $\len(T_k) = \xi(k)(\mu(k) - 5)$. The next lemma provides the final technical detail in this half of our verification. For simplicity of notation, let
\[
	S_k = \bigcup_{1 \leq j \leq \xi(k)} \sigma_j T^j_k.
\]

\begin{lem}\label{lem:TkSkDifference}
	For $k < \omega$ and $\sigma, (\sigma_j)$ as above, we have
	\[
		|K(T_k) - K (\sigma S_k)| \leq O(2^{2^k}).
	\]
\end{lem}

\begin{proof}
	This is an easy ``bit counting'' argument: the number of bits by which $T_k$ and $\sigma S_k$ differ is given by $\len(\sigma) + \sum_j \len(\sigma_j)$. If we also know where the $\sigma_j$'s are located, then we can construct each string from the other. Thus,
	\[
		|K(T_k) - K(\sigma S_k)| \leq K(\sigma) + \sum_{1 \leq j \leq \xi(k)} K(\sigma_j,m_j)
	\]
	omitting constants, where $m_j$ denotes the index at which $\sigma_j$ begins inside $S_k$. We know $\len(\sigma_j) \leq 7$ and $\len(\sigma) \leq -\log(a) + \len(x_k) + 2 = -\log(a) + 2^{2^k} + k + 2$. Further,
	\begin{align*}
		\len(\sigma S_k) &= \len(\sigma) + \sum_{1 \leq j \leq \xi(k)} \len(\sigma_j) + \len(T_k)\\
		&\leq -\log(a) + \len(x_k) + 2 + 7\xi(k) + \xi(k)(\mu(k)-5)\\
		&= -\log(a) + \len(x_k) + 2 +\xi(k)(\mu(k) + 2)
	\end{align*}
	since each of the $\xi(k)$-many blocks codes $\mu(k) - 5$-many bits. Observe that 
	\[
		\xi(k)\mu(k) = 2^{2^k}(2^{2^k} - 1)
	\]
	and hence is of order $2^{2^{k+1}}$. As $m_j \leq \len(S_k)$ we see that $m_j$ is thus at most of order $2^{2^{k+1}}$. But now $K(m_j)$ is at most of order $2^{k+1}$. It is now easily seen that $\sum_j K(\sigma_j,m_j)$ is of order at most $\xi(k)2^{k+1}$, which is $O(2^{2^k})$, as required.
\end{proof}

We can now complete the argument: using the previous lemma, we see
\[
	K(\str{ar}[m]) = K(\sigma S_k) = K(T_k) + O(2^{2^k}).
\]
Further, observe that $\len(T_k)$ is of order $2^{2^{k+1}}$, since $\len(T_k) = \xi(k)(\mu(k) - 5)$. As before, $\lim_{k \rightarrow \infty} \frac{2^{2^k}}{2^{2^{k+1}}} = 0$, and so we may ignore terms of order at most $2^{2^k}$. This allows us to simplify: let $\mathcal{D}$ be the set of $m < \omega$ at which requirement $a = a_i$ has just been attended to. (In other words, $\str{ar}[m] = \sigma S_k$ for some $k$.) Then
\begin{align*}
	\dim(ar) \leq \liminf_{m \in \mathcal{D}} \frac{K(\str{ar}[m])}{m} \leq \liminf_{m \in \mathcal{D}} \frac{K(T_k)}{m} &= \epsilon
\end{align*}
by definition of $T$.

\subsubsection{The dimension of $r$ with respect to $Z$} Recall that 
\[
	Z = Y \oplus \str{\theta}
\]
and that $Y$ contains all $a_i$. As we need to show that $\dim^Z(r) \geq \epsilon$, it does not suffice to exhibit a set of favourable elements, such as our set $\mathcal{D}$ in the previous lemma. Instead, we show we can decode enough elements of $T$ from \emph{any} initial segment of $r$.

Suppose 
\[
	\str{r}[m] = \sigma_1 \cdots \sigma_{k+1} b_1 \cdots b_n \tau
\]
where
\begin{itemize}
	\item $\sigma_{i}$ denotes the initial segment of $r$ that satisfied the stage $i$;
	\item $b_j$ denotes the substring of $r$ that satisfied block $j$ of stage $k+2$; and
	\item $\tau$ is the initial segment of the substring satisfying block $n+1$.
\end{itemize}
Hence, observe we are at stage $k+2$, and $n$ blocks have already been satisfied.

Inside stage $k+1$, the substring $T_k$ has been coded into $ar$. Hence, using the oracle $Z$ which computes all $a_i$, we can recover $T_k$ from $\str{ar}$. Recall that
\[
	\len(T_k) = \xi(k)(\mu(k) - 5).
\]
Observe that since $\lim_{k \rightarrow \infty} \frac{2^{2^k}}{2^{2^{k+1}}} = \lim_{k \rightarrow \infty} \frac{1}{2^k} = 0$, the length of $T_k$ already dwarves the lengths of $T_1 + \ldots + T_{k-1}$; hence, it suffices to compute the blocks saved at stage $k+1$.

Secondly, the worst case to consider above is the case where $n=0$: in that case the initial segment $\sigma_{k+1}$ needs to carry enough information to survive against $\tau$, where $\tau$ is at most of length $\mu(k+1) - 1$. This is not an issue, since $\len(T_k) = \xi(k)(\mu(k) - 5)$ and
\[
	\lim_{k \rightarrow \infty} \frac{\mu(k+1) - 1}{\xi(k)(\mu(k) - 5)} = 0.
\]
Hence, the information provided in $T_k$ dwarves the unfinished block $\tau$. It now suffices to show that $T_k$ and the completely coded substrings $T^1_{k+1},\ldots,T^n_{k+1}$ can be easily recovered from $\str{ar}[m]$. This follows from an argument analogous to \cref{lem:TkSkDifference}:
\begin{itemize}
	\item take a machine that trims $\str{r}$ to length $\nu(k+1) - 1$, and denote the resultant string by $\rho$ (this is where stage $k + 1$ has just been completed);
	\item compute the correct projection factor $a_i = a$ for stage $k + 1$ using $Z$ (and from the Cantor pairing function);
	\item compute the largest dyadic interval in $a[\real{\rho}]$, and let $d$ denote its left end-point. Now
	\[
		\str{d} = \sigma S_k \sigma'
	\]
	where $\len(\sigma') \leq 7$, by the cost lemma \ref{lem:atMostSeven}.
	\item By the previous \cref{lem:TkSkDifference}, we know that the complexity of isolating $T_k$ from $\sigma S_k \sigma'$ is not significant, as required. An identical argument recovers the $n$ blocks.
\end{itemize}

It now follows from \cref{lem:TkSkDifference} that
\begin{align}
	K^Z(T_k \cup T^1_{k+1} \cup \ldots \cup T^n_{k+1}) &\leq K^Z(\str{r}[m]) + O(2^{2^k}) + O(n2^{k+2}) \nonumber
	\intertext{where $n < \xi(k+1)$. Thus, in particular}
	\frac{K^Z(T_k \cup T^1_{k+1} \cup \ldots \cup T^n_{k+1})}{m} &\leq \frac{K^Z(\str{r}[m])}{m} + \frac{O(2^{2^k}) + O(n2^{k+2})}{m} \label{eq:theFinalStep}
\end{align}
where $m = \nu(k+1) + n\mu(k+1) + \len(\tau)$ and $n < \xi(k+1)$. Next, we verify that the length of $T$ computed on the left-hand side of \cref{eq:theFinalStep} is sufficiently long: indeed,
\begin{align*}
	|m - \len(T_k \cup T^1_{k+1} \cup \ldots \cup T^n_{k+1})| &= \len(\tau) + \nu(k) +1 + 5\xi(k) + 5n.
\end{align*}
Now, $m = \nu(k+1) + n\mu(k+1) + \len(\tau)$ and $n < \xi(k+1)$ imply
\[
	\frac{\len(\tau) + \nu(k) + 5\xi(k) + 5n}{m} \leq \frac{\len(\tau) + \nu(k) + 5(\xi(k) + \xi(k+1))}{\len(\tau) + \nu(k+1) + \xi(k+1)\mu(k+1)};
\]
applying limits as $k$ goes to infinity shows that the term vanishes. Going back and applying $\liminf$ to both sides of \cref{eq:theFinalStep} now proves that its left-hand side equals $\epsilon$.\\

Finally, since $m$ is of order $\nu(k+1) + n\mu(k+1)$, i.e.\ of order at least $2^{2^{k+1}}$, the right-hand side of \cref{eq:theFinalStep} simplifies to its first term. Putting it all together and applying $\liminf$ we hence obtain
\begin{align*}
	\epsilon &= \liminf_{k \rightarrow \infty} \frac{K^Z(T_k \cup T^1_{k+1} \cup \ldots \cup T^n_{k+1})}{m} \leq \frac{K^Z(\str{r}[m]}{m} = \dim^Z(r)
\end{align*} 
as required. \Cref{thm:secondThm} now follows immediately from the same arguments as the proof of \cref{thm:firstThm}, and the overview we gave at the start of this section.

\bibliographystyle{abbrv}
\bibliography{marstrandsThm.bib}

\begin{thebibliography}{10}

\bibitem{athreyaHitchcockEtAl}
K.~B. Athreya, J.~M. Hitchcock, J.~H. Lutz, and E.~Mayordomo.
\newblock Effective strong dimension in algorithmic information and
  computational complexity.
\newblock {\em SIAM Journal on Computing}, 37(3):671--705, 2007.

\bibitem{dimensionZero}
V.~Beresnevich, K.~Falconer, S.~Velani, and A.~Zafeiropoulos.
\newblock Marstrand's theorem revisited: Projecting sets of dimension zero.
\newblock {\em Journal of Mathematical Analysis and Applications},
  472(2):1820--1845, 2019.

\bibitem{falconerEtAl}
V.~Beresnevich, K.~Falconer, S.~Velani, and A.~Zafeiropoulos.
\newblock Marstrand's theorem revisited: projecting sets of dimension zero.
\newblock {\em J. Math. Anal. Appl.}, 472(2):1820--1845, 2019.
\newblock With an appendix by David Simmons, Han Yu and Zafeiropoulos.

\bibitem{birkhoffRota}
G.~Birkhoff and G.-C. Rota.
\newblock {\em Ordinary differential equations}.
\newblock John Wiley \& Sons, New York-Chichester-Brisbane, third edition,
  1978.

\bibitem{caihart}
J.-Y. Cai and J.~Hartmanis.
\newblock On {H}ausdorff and topological dimensions of the {K}olmogorov
  complexity of the real line.
\newblock {\em J. Comput. System Sci.}, 49(3):605--619, 1994.

\bibitem{caseLutz}
A.~Case and J.~H. Lutz.
\newblock Mutual dimension.
\newblock {\em ACM Trans. Comput. Theory}, 7(3):Art. 12, 26, 2015.

\bibitem{chaitin}
G.~J. Chaitin.
\newblock A theory of program size formally identical to information theory.
\newblock {\em J. Assoc. Comput. Mach.}, 22:329--340, 1975.

\bibitem{daviesCounterexample}
R.~O. Davies.
\newblock Two counterexamples concerning {H}ausdorff dimensions of projections.
\newblock {\em Colloq. Math.}, 42:53--58, 1979.

\bibitem{downey}
R.~G. Downey and D.~R. Hirschfeldt.
\newblock {\em Algorithmic randomness and complexity}.
\newblock Theory and Applications of Computability. Springer, New York, 2010.

\bibitem{edgar}
G.~Edgar.
\newblock {\em Measure, topology, and fractal geometry}.
\newblock Undergraduate Texts in Mathematics. Springer, New York, second
  edition, 2008.

\bibitem{erdosEtAl}
P.~Erd\H{o}s, K.~Kunen, and R.~D. Mauldin.
\newblock Some additive properties of sets of real numbers.
\newblock {\em Fund. Math.}, 113(3):187--199, 1981.

\bibitem{falconerFraserXiong}
K.~Falconer, J.~Fraser, and X.~Jin.
\newblock Sixty years of fractal projections.
\newblock In {\em Fractal geometry and stochastics {V}}, volume~70 of {\em
  Progr. Probab.}, pages 3--25. Birkh\"{a}user/Springer, Cham, 2015.

\bibitem{strongMarstrand}
K.~Falconer and P.~Mattila.
\newblock Strong {M}arstrand theorems and dimensions of sets formed by subsets
  of hyperplanes.
\newblock {\em J. Fractal Geom.}, 3(4):319--329, 2016.

\bibitem{fortnow}
L.~Fortnow.
\newblock Kolmogorov complexity.
\newblock In {\em Aspects of complexity ({K}aikoura, 2000)}, volume~4 of {\em
  De Gruyter Ser.\ Log.\ Appl.}, pages 73--86. de Gruyter, Berlin, 2001.

\bibitem{hitchcockPhD}
J.~M. Hitchcock.
\newblock {\em Effective fractal dimension: foundations and applications}.
\newblock PhD thesis, Iowa State University, USA, 2003.

\bibitem{hitchcockPaper}
J.~M. Hitchcock.
\newblock Correspondence principles for effective dimensions.
\newblock {\em Theory Comput. Syst.}, 38(5):559--571, 2005.

\bibitem{packingNoMars}
M.~J\"{a}rvenp\"{a}\"{a}.
\newblock On the upper {M}inkowski dimension, the packing dimension, and
  orthogonal projections.
\newblock {\em Ann. Acad. Sci. Fenn. Ser. A I Math. Dissertationes}, (99):34,
  1994.

\bibitem{kaufman}
R.~Kaufman.
\newblock On {H}ausdorff dimension of projections.
\newblock {\em Mathematika}, 15:153--155, 1968.

\bibitem{kolmogorov}
A.~N. Kolmogorov.
\newblock Three approaches to the definition of the concept ``quantity of
  information''.
\newblock {\em Problemy Pereda\v{c}i Informacii}, 1(vyp. 1):3--11, 1965.

\bibitem{levin}
L.~A. Levin.
\newblock The concept of a random sequence.
\newblock {\em Dokl. Akad. Nauk SSSR}, 212:548--550, 1973.

\bibitem{livitanyi}
M.~Li and P.~Vit{\'a}nyi.
\newblock {\em An introduction to {K}olmogorov complexity and its
  applications}.
\newblock Texts in Computer Science. Springer, Cham, 2019.
\newblock Fourth edition of [ MR1238938].

\bibitem{lutzGales}
J.~H. Lutz.
\newblock Gales and the constructive dimension of individual sequences.
\newblock In {\em Automata, languages and programming ({G}eneva, 2000)}, volume
  1853 of {\em Lecture Notes in Comput. Sci.}, pages 902--913. Springer,
  Berlin, 2000.

\bibitem{lutzGalesNew}
J.~H. Lutz.
\newblock Dimension in complexity classes.
\newblock {\em SIAM J. Comput.}, 32(5):1236--1259, 2003.

\bibitem{lutz2003}
J.~H. Lutz.
\newblock The dimensions of individual strings and sequences.
\newblock {\em Inform. and Comput.}, 187(1):49--79, 2003.

\bibitem{kakeya}
J.~H. Lutz and N.~Lutz.
\newblock Algorithmic information, plane {K}akeya sets, and conditional
  dimension.
\newblock {\em ACM Trans. Comput. Theory}, 10(2):Art. 7, 22, 2018.

\bibitem{marsExt}
J.~H. Lutz, N.~Lutz, and E.~Mayordomo.
\newblock Extending the reach of the point-to-set principle.
\newblock In {\em 39th {I}nternational {S}ymposium on {T}heoretical {A}spects
  of {C}omputer {S}cience}, volume 219 of {\em LIPIcs. Leibniz Int. Proc.
  Inform.}, pages Art. No. 48, 14. Schloss Dagstuhl. Leibniz-Zent. Inform.,
  Wadern, 2022.

\bibitem{lutzMayor2}
J.~H. Lutz and E.~Mayordomo.
\newblock Dimensions of points in self-similar fractals.
\newblock {\em SIAM J. Comput.}, 38(3):1080--1112, 2008.

\bibitem{effectiveDimension}
N.~Lutz and D.~M. Stull.
\newblock Projection theorems using effective dimension.
\newblock In {\em 43rd {I}nternational {S}ymposium on {M}athematical
  {F}oundations of {C}omputer {S}cience}, volume 117 of {\em LIPIcs. Leibniz
  Int. Proc. Inform.}, pages Art. No. 71, 15. Schloss Dagstuhl. Leibniz-Zent.
  Inform., Wadern, 2018.

\bibitem{lutzStullPointsonLine}
N.~Lutz and D.~M. Stull.
\newblock Bounding the dimension of points on a line.
\newblock {\em Inform. and Comput.}, 275:104601, 15, 2020.

\bibitem{marstrand}
J.~M. Marstrand.
\newblock Some fundamental geometrical properties of plane sets of fractional
  dimensions.
\newblock {\em Proc. London Math. Soc. (3)}, 4:257--302, 1954.

\bibitem{martinloef}
P.~Martin-L\"{o}f.
\newblock The definition of random sequences.
\newblock {\em Information and Control}, 9:602--619, 1966.

\bibitem{mattila}
P.~Mattila.
\newblock Hausdorff dimension, orthogonal projections and intersections with
  planes.
\newblock {\em Ann. Acad. Sci. Fenn. Ser. A I Math.}, 1(2):227--244, 1975.

\bibitem{mayordomo}
E.~Mayordomo.
\newblock A {K}olmogorov complexity characterization of constructive
  {H}ausdorff dimension.
\newblock {\em Inform. Process. Lett.}, 84(1):1--3, 2002.

\bibitem{mediniVid}
A.~Medini and Z.~Vidny{\'a}nszky.
\newblock Zero-dimensional $\sigma$-homogeneous spaces, 2021.

\bibitem{miller}
A.~W. Miller.
\newblock Infinite combinatorics and definability.
\newblock {\em Ann. Pure Appl. Logic}, 41(2):179--203, 1989.

\bibitem{ryabko1}
B.~Y. Ryabko.
\newblock Coding of combinatorial sources and {H}ausdorff dimension.
\newblock {\em Dokl. Akad. Nauk SSSR}, 277(5):1066--1070, 1984.

\bibitem{ryabko2}
B.~Y. Ryabko.
\newblock Noise-free coding of combinatorial sources, {H}ausdorff dimension and
  {K}olmogorov complexity.
\newblock {\em Problemy Peredachi Informatsii}, 22(3):16--26, 1986.

\bibitem{schnorr}
C.-P. Schnorr.
\newblock A unified approach to the definition of random sequences.
\newblock {\em Math. Systems Theory}, 5:246--258, 1971.

\bibitem{soare}
R.~I. Soare.
\newblock {\em Turing computability}.
\newblock Theory and Applications of Computability. Springer-Verlag, Berlin,
  2016.
\newblock Theory and applications.

\bibitem{solomonoff}
R.~J. Solomonoff.
\newblock A preliminary report on a general theory of inductive inference.
\newblock Zator Company Cambridge, MA, 1960.

\bibitem{staiger}
L.~Staiger.
\newblock Kolmogorov complexity and {H}ausdorff dimension.
\newblock {\em Inform. and Comput.}, 103(2):159--194, 1993.

\bibitem{zoltan}
Z.~Vidny{\'a}nszky.
\newblock Transfinite inductions producing coanalytic sets.
\newblock {\em Fund. Math.}, 224(2):155--174, 2014.

\end{thebibliography}

\end{document}